\documentclass[12pt, reqno, a4paper]{amsart}
\usepackage{amsmath,mathrsfs}
\usepackage{amssymb}
\usepackage{calligra}
\usepackage{dsfont}
\usepackage{bbm}
\usepackage{mathtools}
\usepackage{appendix}
\usepackage{autobreak}
\usepackage{todonotes}
\usepackage{enumitem}

\usepackage{pdfsync}

\newcommand\NoBlackBoxes{\global\overfullrule0pt}
\NoBlackBoxes

\newtheorem{definition}{Definition}[section]
\newtheorem{theorem}[definition]{Theorem}
\newtheorem{prop}[definition]{Proposition}
\newtheorem{lem}[definition]{Lemma}
\newtheorem{rem}[definition]{Remark}

\newtheorem{cor}[definition]{Corollary}

\newcommand{\1}{\mathbbm{1}}
\newcommand{\N}{{\mathbb{N}}}
\newcommand{\R}{{\mathbb{R}}}

\renewcommand{\P}{{\mathbb{P}}}
\newcommand{\E}{\mathbb{E}}
\newcommand{\V}{\mathbb{V}}
\newcommand{\CUn}{\mathcal{U}_n}
\newcommand{\vep}{\varepsilon}
\newcommand{\eps}{\varepsilon}
\newcommand{\nconv}{\xrightarrow[]{n\to\infty}}
\newcommand{\Pconv}{\xrightarrow[n\to\infty]{\P}}
\newcommand{\dconv}{\xRightarrow{n\to\infty}}

\newcommand{\EW}[1]{\mathbb{E}\left[#1\right]}
\newcommand{\ind}{\operatorname{1}}

\newcommand{\WK}[1]{\mathbb{P}\left(#1\right)}

\begin{document}
\title[A CLT for the hitting time for a random walk on a random graph]{A Central Limit Theorem for the mean starting hitting time
for a random walk on a random graph}

\author[Matthias L\"owe]{Matthias L\"owe}\thanks{Research of the both authors was
funded by the Deutsche Forschungsgemeinschaft (DFG, German Research Foundation) under Germany's Excellence Strategy
EXC 2044-390685587, Mathematics M\"unster: Dynamics-Geometry-Structure}
\address[Matthias L\"owe]{Fachbereich Mathematik und Informatik,
Universit\"at M\"unster,
Einsteinstra\ss e 62,
48149 M\"unster,
Germany}

\email[Matthias L\"owe]{maloewe@uni-muenster.de}

\author[Sara Terveer]{Sara Terveer}
\address[Sara Terveer]{Fachbereich Mathematik und Informatik,
Universit\"at M\"unster,
Einsteinstra\ss e 62,
48149 M\"unster,
Germany}

\email[Sara Terveer]{sara.terveer@uni-muenster.de}



\begin{abstract}
We consider simple random walk on a realization of an Erd\H{o}s-R\'enyi graph that is asymptotically almost surely (a.a.s.) connected. We show a Central Limit Theorem (CLT) for the average starting hitting time, i.e.\ the expected time it takes the random walker on average to first hit a  vertex $j$ when starting in a fixed vertex $i$. The average is taken with respect to $\pi_j$, the invariant measure of the random walk.
\end{abstract}

\maketitle
\section{Introduction}
Random walks on random graphs have been actively studied over the past decades, see e.g.
\cite{doylesnell, lovaszgraphs, woess,BLPS18,grimmett}, among many others.
The two layers of randomness, one for the random graph, the other one for the random walk on it, are typical of many fascinating problems in modern probability, such as spin glasses, random walks in a environment, and others. Moreover, random walks on random graphs also reflect typical characteristics of the underlying random graph. Therefore, random walks are a key tool to understand the properties of random graphs close to the point of
phase transition (see the recent monograph \cite{vdH17}).

There are various quantities to describe the behaviour of a random walk on a random graph, most of them related to the question, how quickly the random walk is able to see various areas of the graph. The mixing time characterizes how long it takes the distribution of a walker to get close to its equilibrium distribution, the hitting and commute times represent the time it takes to get from one vertex to another, and the cover time describes the time a walker needs to see the entire graph. In this note we will investigate the hitting time of random walk on random graphs.

To define it, assume we are given a realization $G=(V,E)$ of an Erd\H{o}s-R\'enyi graph $\mathcal{G}(n+1,p)$, i.e.\ we choose a random graph $G$ on $V=\{1, \ldots, n+1\}$ such that all undirected edges $e=\{i,j\}\in E$ are realized independently with the same probability $p$. Note the slightly unconventional choice of $n+1$ vertices rather than $n$. Indeed, later it will be convenient to think of every vertex to have $n$ potential neighbors.
Also note that $p=p_n$ may and typically will depend on $n$ in such a way
that $p_n \gg \frac{\log n}{n}$, by which we mean that $\frac{\log n}{n p_n} \to 0$ as $n \to \infty$. Hence with probability converging to one $G$ is connected and all the probabilities below will be understood conditionally on the event that $G$ is connected (see \cite[Theorem 7.3]{bollobas}).

While the above setup suffices to define $G$ for fixed $n$ it does not allow to consider almost sure results or limits for $n \to \infty$, that we will need later on (see e.g.\ Section 6). To this end, we will construct the entire sequence of random graphs as follows. Given the sequence $(p_n)$ let us consider $i\neq j \in \N$. Let $\vep_{i,j}(n+1),i,j=1,\dots,n $ be the indicator for the event that the undirected edge $\{i,j\}$ is present in the random graph on $V=\{1, \ldots, n+1\}$. We will assume that the $\vep_{i,j}(n+1)$ behave like a time-inhomogeneous Markov chain. More precisely:
\begin{align*}
&\P(\vep_{i,j}(n+1)=0 |\vep_{i,j}(n)=0 )=1\\
&\P(\vep_{i,j}(n+1)=1 |\vep_{i,j}(n)=0 )=0\\
&\P(\vep_{i,j}(n+1)=0 |\vep_{i,j}(n)=1 )=1-q(n)\\
&\P(\vep_{i,j}(n+1)=1 |\vep_{i,j}(n)=1 )=q(n)
\end{align*}
where we have set $q(n)=p_n/p_{n-1}$. This setup of course only works, if $(p_n)$ is non-increasing.
However, if $(p_n)$ is increasing, we can construct the complement of $G$ (which we denote by $G^c$) in the same fashion. More precisely, $G^c$ is the graph on $\{1, \dots, n+1\}$ that contains an edge $e$, if and only if $e \notin E$. Of course, being able to construct $G^c$ also implies that we can construct $G$.

On a realization of $G$ on a fixed number of vertices $n+1$ consider simple random walk in discrete time $(X_t)$: If $X_t$ is in $v\in V$ at time $t$, $X_{t+1}$ will be in $w$ with probability $1/d_v$ ($d_v$ denoting the degree of $v$), if $\{v,w\} \in E$ and with probability $0$, otherwise. This random walk has an invariant distribution given by
$$
\pi_i:= \frac{d_i}{\sum_{j \in V}d_j}=\frac{d_i}{2 |E|}.
$$
Let 
$H_{ij}$ be the expected time it takes the walk to reach vertex $j$ when starting from vertex $i$.
Moreover, let
\begin{equation}\label{def H}
H_j :=  \sum_{i\in V} \pi_i H_{ij} \quad \mbox{ and}  \quad H^i :=  \sum_{j \in V} \pi_j H_{ij}
\end{equation}
be the {\it average target hitting time} and {\it the average starting hitting time}, respectively (these names are taken from \cite{levinperes}).
Note that $H_j$ and $H^i$ are expectation values in the random walk measure. However, with respect to the realization of the random graph they are random variables. In general, $H_j$ and $H^i$ will be different.

In \cite{LT14} the asymptotic behaviour of $H_j$ and $H^i$ was analyzed. Confirming a prediction in the physics literature \cite{Sood}, it was shown that
\begin{equation}\label{LLN}
H_j=n(1+o(1)) \qquad \mbox{as well as} \quad H^i=n(1+o(1))
\end{equation}
a.a.s. By the latter expression we mean that for a given vertex $i$ the probability that $H_i$ or $H^i$ are not of the order prescribed by \eqref{LLN} goes to 0, as $n \to \infty$. The analogous result for hypergraphs was proved in \cite{HL19}.

Equation \eqref{LLN} can be considered a law of large numbers for the random variables $H_j$ and $H^i$. In this note we will study the fluctuations around this law of large numbers for $H^i$ on the level of a Central Limit Theorem (CLT, for short).
It will turn out that we are able to rewrite $H^i$, up to negligible terms, as a incomplete U-statistic with a kernel of degree $2$. However, quite different from the usual situation, the kernel will not depend on a sequence of i.i.d.\ random variables, but on a triangular array of weakly dependent random variables. Moreover, the ''incompleteness'' is random as well, and the random variables
determining which terms of the U-statistic are present, depend on the other random variables in the U-statistic.
We will need to prove a CLT for this setup, first. Its proof will rely on a martingale method introduced in \cite{girko1990} and an application thereof presented in \cite{LoTe20a}.
More precisely, we will show
\begin{theorem}
\label{main_theo}
Let $H^i$ be defined as \eqref{def H}, $\theta_n$ be given by \eqref{defn}, and $\mu$ be defined as in \eqref{eq:defmu}.
Moreover, assume that $p_n\in (0,1)$ is either an increasing or a decreasing sequence of probabilities. We denote $\lim p_n=: p^* \in [0,1]$ and we furthermore assume that $n p_n/\log n \to \infty$ and that $n(1-p_n)\to \infty$.
Then, with the above setting
\begin{equation}
\frac{H^i-(n-2)-2\mu_n^2\binom{n+1}{2}p_n}{2n\theta_n}\dconv\mathcal{N}\left(0,1+\frac{(1-p^*)p^*}{2}\right),
\end{equation}
where $\mathcal{N}\left(m, s^2\right)$ denotes a normal distribution with expectation $m$ and variance $s^2$ and
$\dconv$ denotes convergence in distribution.
\end{theorem}

The rest of the paper is organized in the following way. In Section 2 we will decompose $H^i$ in terms of the spectrum of a variant of the Laplace matrix of the random graph. This decomposition is the key to rewriting $H^i$ in terms of a U-statistic $U_n$ plus smaller order terms in Section 3. In Section 4 we recall a CLT for a incomplete U-statistic over a triangular array proved in \cite{LoTe20a}. In Section 5 we apply this CLT to prove a CLT for a variant for $U_n$ with extra assumptions. Sections 6 and 7 serve to remove these assumptions to obtain a CLT for $U_n$. Finally, in Section 8 we will see that $U_n$ is indeed the central term in the decomposition of $H^i$, i.e.\ we will see that Theorem \ref{main_theo} is true. In an appendix we will collect some auxiliary results.

\section{Spectral decomposition of the hitting times}
Our starting point will be the same as in \cite{LT14} or \cite{HL19}, the spectral decomposition of the hitting times, taken from \cite[Section 3]{lovaszgraphs}. To introduce it, let $A$ be the adjacency matrix of $G$, i.e.\ $A=(a_{ij})$, with $a_{ij}=\ind_{\{i,j\}\in E}$. Let $D$ be the diagonal matrix $D=(\mathrm{diag}(d_i))_{i=1}^{n+1}$. With $A$ we associate the matrix $B$ defined as $B:=D^{-\frac 12}AD^{-\frac 12}$. The matrix $B$ is intrinsically related to the Laplace matrix $L$ of $G$ defined as $L:=\operatorname{Id}-B$. Note that $B=(b_{ij})$ with $b_{ij}=\frac{a_{ij}}{\sqrt{d_i d_j}}$. Therefore, $B$ is symmetric and hence has real eigenvalues.
$\lambda_1 = 1$ is an eigenvalue of $B$, since the vector $w:=(\sqrt{d_1}, \cdots, \sqrt{d_n})$ satisfies $Bw=w$.
By the Perron-Frobenius theorem $\lambda_1$ is the largest eigenvalue.
We order the eigenvalues $\lambda_k$ of $B$ such that $\lambda_1 \geq \lambda_2 \geq \cdots \geq \lambda_{n+1} $
and we normalize the eigenvectors $v_k$ to the eigenvalues $\lambda_k$ to have length one. Thus, in particular,
$$
v_1 := \frac{w}{\sqrt{2|\widetilde{E}|}}= \left(\sqrt{\frac{d_j}{2|\widetilde{E}|}} \right)_{j=1}^{n+1}.
$$
Also recall that the matrix of the eigenvector is orthogonal and the scalar product of two eigenvectors $v_i$ and $v_j$  satisfies $\langle v_i, v_j\rangle =\delta_{ij}$.
With this notation we can give the following representation of 
the average starting hitting time.
\begin{prop}[cf. {\cite[Theorem 3.1 and Formula (3.3)]{lovaszgraphs}}]
For all $i\neq j \in V$ we have
$$
H_{ij}=2 |\widetilde{E}| \sum_{k=2}^{n+1} \frac{1}{1- \lambda_k} \left(   \frac{v_{k, j}^2}{d_j}   -   \frac{v_{k, i}  v_{k, j}}{\sqrt{d_i d_j}}     \right).
$$
Thus,
\begin{equation}\label{repr_starting}
H^i=   \sum_{k=2}^{n+1} \frac{1}{1-\lambda_k}
\end{equation}
\end{prop}
Formula \eqref{repr_starting} will be the basis our fluctuations analysis in the next sections.

\section{Rewriting $H^i$}
In this section we will rewrite $H^i$ in terms of a sum of a U-statistic and a negligible term. This will be the starting point to prove our main result Theorem \ref{main_theo}.

Consider \eqref{repr_starting}. Since $\lambda_k<1$ for $k\geq2$ we can apply the geometric series to write
\begin{align}
H^i&=\sum\limits_{k=2}^{n+1}\sum\limits_{l=0}^\infty\lambda_k^l=\sum\limits_{k=2}^{n+1} \left(1+\lambda_k+\lambda_k^2+O(\lambda_k^3)\right)
=n+\sum\limits_{k=2}^{n+1}\lambda_k+\sum\limits_{k=2}^{n+1}\lambda_k^2+O\left(\sum\limits_{k=2}^{n+1}\lambda_k^3\right)
\notag\\&
=(n-2)+\sum\limits_{k=1}^{n+1}\lambda_k^2+O\left(\sum\limits_{k=2}^{n+1}\lambda_k^3\right)\label{eq:Hirepres},
\end{align}
since $B$ has trace $0$ and $\lambda_1=1$.
We will show later that $O\left(\sum\limits_{k=2}^{n+1}\lambda_k^3\right)$ is negligible on the scale of a CLT. Therefore consider
\begin{align}\label{Ustat1}
\sum\limits_{k=1}^{n+1}\lambda_k^2=\mathrm{tr}(B^2)&=\sum\limits_{i,j=1}^{n+1}\frac{a_{ij}a_{ji}}{d_id_j}
=\sum\limits_{i,j=1}^{n+1}\frac{a_{ij}}{d_id_j}
=2\sum\limits_{i<j}\frac{a_{ij}}{d_id_j}
\end{align}
which follows from $a_{ij}=a_{ji}$ and $a_{ij}\in\{0,1\}$. 
$$
\sum\limits_{i<j}\frac{a_{ij}}{d_id_j}=\sum\limits_{i<j}\1_{\{a_{ij}=1\}}\frac{a_{ij}}{d_id_j}
=\sum\limits_{i<j}a_{ij}\frac{1}{\tilde{d}_i^j+1}\frac{1}{\tilde{d}_j^i+1},
$$
where for vertices $i,j_1,\dots,j_K$ (and constant $K \in \N$), we define
\begin{equation}
\tilde{d}_i^{(j_1,\dots,j_K)}\coloneqq\sum\limits_{\substack{l=1\\l\notin\{i, j_1,\dots,j_K\}}}^{n+1}a_{il}\qquad\text{ and }\qquad\tilde{d}_i^j\coloneqq\tilde{d}_i^{(j)}.
\label{eq:tilded}
\end{equation}
Notice that $\tilde{d}_i^j$ is a random variable and binomially distributed with parameters $n-1$ and $p$.
Thus, with
\begin{equation}\label{eq:defmu}
\mu\coloneqq\mu_n\coloneqq \EW{\frac{1}{\tilde{d}_i^j+1}}
\end{equation}
we are led to consider the U-statistic
\begin{equation}
U_n=
\frac{1}{n\theta_n}\sum\limits_{1\leq i<j\leq n+1}\left(\frac{a_{i,j}}{{d}_i{d}_j}-\mu^2p\right)=\frac{1}{n\theta_n}\sum\limits_{1\leq i<j\leq n+1}\left(\frac{a_{i,j}}{(\tilde{d}_i^j+1)(\tilde{d}_j^i+1)}-\mu^2p\right).
\label{eq:mainustat}
\end{equation}
Here $\theta_n$ is a normalizing sequence we will specify in \eqref{defn}.
Note, however that $U_n$ has an unusual structure for a U-statistic, due to the lack of independence
of the $a_{i,j}$ and the $\tilde{d}_i^j$ as well as between the $\tilde{d}_i^j$ themselves.

To begin with, we will prove that the expectation of inverse moments of $d_i$ is asymptotically insensitive to removing a constant number of vertices.
In what follows, for two sequence $a_n, b_n$ we will write $a_n\approx b_n$,$a_n\lesssim b_n$, and $a_n \gtrsim b_n$, respectively, to denote that $a_n=b_n(1+o(1))$, $a_n= O(b_n)$, and $a_n= \Omega(b_n)$, respectively, in the Landau notation.
We will show
\begin{prop}\label{prop:asymptindep}
For fixed, constant numbers $l,k,K$ and vertices $i,j_1,\dots,j_l$,
\[\E\Bigl[\big(\frac{1}{\tilde{d}_i^{(j_1,\dots,j_l)}+K+1}\bigr)^k\Bigr]
\approx\E\Bigl[\big(\frac{1}{X+1}\big)^k\Bigr]
=\frac{1}{(np)^k}\left(1+o(1)\right)=\Theta\left((np)^{-k})\right)\]
Here, $X\sim\mathrm{Bin}(n,p)$ is a binomially distributed random variable with parameters $n$ and $p$ (in fact, we will more often make use of the same result for $X\sim\mathrm{Bin}(n-1,p)$).
\end{prop}
\begin{proof}
Clearly, $$\E\Bigl[\Bigl(\frac{1}{\tilde{d}_i^{(j_1,\dots,j_l)}+K+1}\Bigr)^k\Bigr]\leq
\E\Bigl[\Bigl(\frac{1}{\tilde{d}_i^{(j_1,\dots,j_l)}+1}\Bigr)^k\Bigr].$$

%

On the other hand, by Jensen's inequality
\begin{multline*}
\E\Bigl[\Bigl(\frac{1}{\tilde{d}_i^{(j_1,\dots,j_l)}+K+1}\Bigr)^k\Bigr]
\geq \Bigl(\E\Bigl[\tilde{d}_i^{(j_1,\dots,j_l)}+K+1\Bigr]\Bigr)^{-k}\\
=\Bigl(\frac{1}{(n-l)p+K+1}\Bigr)^k \geq\Bigl(\frac{1}{np+K+1}\Bigr)^k= \frac{1}{(np)^k}\left(1-o(1)\right)
\end{multline*}
because we know that $\tilde{d}_i^{(j_1,\dots,j_l)}\sim \mathrm{Bin}(n-l,p)$. Finally, by Remark \ref{rem:zni}
$$\E\Bigl[\big(\frac{1}{X+1}\big)^k\Bigr]
=\frac{1}{(np)^k}\left(1+o(1)\right).$$
The fact that $\frac{1}{(np)^k}\left(1+o(1)\right)=\Theta\left((np)^{-k}\right)$ concludes the proof.
\end{proof}
\section{A Central Limit Theorem for U-statistics over a triangular array of random variables}
In the sequel we will use the following CLT for U-statistics over a triangular array of random variables that was proved in \cite{LoTe20a}. There we considered the following setting:

Let $X_{n1},\dots,X_{nn}$ be a triangular array of random variables with values in $\R$, independent of each other and having the same distribution function $F_n(x)$ in each row.
Moreover, let
$h_n: \R \times \R \to \R$ be some real-valued, symmetric, integrable Borel function, satisfying that for all $n$ and $1 \le i \neq j \le n$
\begin{equation}
\EW{h_n(X_{ni},X_{nj})}=0 \qquad \mbox{and } \quad \EW{h_n^2(X_{ni},X_{nj})}<\infty.
\label{eq:hmoments}
\end{equation}
Finally $Z_{ij}=Z_{ji}$ are assumed to be i.i.d.\ $\mathrm{Ber}(p_n)$ random variables (apart from the constraint that $Z_{ij}=Z_{ji}$), independent of the $(X_{ni})_{i,n}$.
For $i,j=1,\dots,n$, consider
\[\Phi_n(i,j)=Z_{ij}\cdot h_n(X_{ni},X_{nj})\]
Let us construct the following U-statistic
\[\mathcal{U}_n=\binom{n}{2}^{-1}\sum\limits_{1\leq i<j\leq n}\Phi_n(i,j)=\binom{n}{2}^{-1}\sum\limits_{1\leq i<j\leq n}Z_{i,j}\cdot h_n(X_{ni},X_{nj}).\]
Using a Hoeffding-type decomposition (see \cite{Hoe48}) we introduce
\begin{align}
\begin{split}
\Psi_j^{(n)}(i)\coloneqq\EW{\Phi_n(i,j)\mid X_i,Z_{ij}}
=Z_{ij}\EW{h_n(X_{ni},X_{nj})\mid X_{ni}}\end{split}\label{Psi}\\
\begin{split}
\beta_n^2&\coloneqq\EW{\Phi_n^2(1,2)}, \quad \gamma_n^2\coloneqq\EW{\left(\Psi_2^{(n)}(1)\right)^2}, \mbox{ and } \theta_n^2\coloneqq np_n\gamma_n^2+\beta_n^2/2.
\label{defn}
\end{split}
\end{align}
Also define
\begin{align}
\tilde{\Phi}_n(i,j)&=\Phi_n(i,j)-\Psi_j^{(n)}(i)-\Psi_i^{(n)}(j),\notag\\
\tilde{h}_n(X_{ni},X_{nj})&=h_n(X_{ni},X_{nj})-\EW{h_n(X_{ni},X_{nj})\mid X_{ni}}-\EW{h_n(X_{ni},X_{nj})\mid X_{nj}}.\notag
\end{align}
One checks that $\tilde{\Phi}_n(i,j)=Z_{i,j}\tilde{h}_n(X_{ni},X_{nj}),\label{eq:tildephi}$
again has expectation zero and even
\begin{equation}
\EW{\tilde{h}_n(X_{ni},X_{nj})\mid X_{nk}}=0,
\label{eq:htildecentering}
\end{equation}
holds true for $i\neq j$ and every $k$, also for $k=i,j$.

In what follows we will drop the index $n$ whenever suitable. Moreover, we will write $h(i,j)$ and $\tilde{h}(i,j)$ as abbreviations for $h_n(X_{ni},X_{nj})$ and $\tilde{h}_n(X_{ni},X_{nj})$.
One can show that
$V \CUn \approx 4n^{-2} \theta_n^2$, i.e.\ $\sqrt{\V\CUn}\binom{n}{2}\approx n\theta_n$ (see \cite[Lemma 1.3]{LoTe20a}). 
\begin{align}
{G}_k(i,j)&=\EW{{\Phi}(i,k){\Phi}(j,k)\mid X_i,X_j,Z_{ik},Z_{jk}}
\eqqcolon Z_{ik}Z_{jk}{H}(i,j)
\label{eq:G}\\
\tilde{G}_k(i,j)&=\EW{\tilde{\Phi}(i,k)\tilde{\Phi}(j,k)\mid X_i,X_j,Z_{ik},Z_{jk}}
\eqqcolon Z_{ik}Z_{jk}\tilde{H}(i,j).
\label{eq:tildeG}
\end{align}
Then,

\begin{theorem}\label{theo2} (see \cite[Theorem 1.5]{LoTe20a})
Assume that for every $\vep >0$
\begin{align}
\frac{1}{n\theta_n^2}\E\Bigl[\bigl(\sum\limits_{j=2}^{n}\Psi_j(1)\bigr)^2\1_{\bigl\{\bigl|\sum\limits_{j=2}^{n}\Psi_j(1)\bigr|\geq\eps\theta_n n\bigr\}}\Bigr]&\nconv 0\label{6}\tag{C1}\\
\theta_n^{-2}\EW{\tilde{\Phi}^2(1,2)\1_{\left\{|\tilde{\Phi}(1,2)|\geq \eps\theta_n n\right\}}}&\nconv 0\label{20}\tag{C2}\\
p\,\theta_n^{-2}\E\bigl[\tilde{H}(1,1)\1_{\bigl\{|\tilde{H}(1,1)|\geq \frac{\eps\theta_n^2 n}{p}\bigr\}}\bigr]&\nconv 0\label{21}\tag{C3}\\
\theta_n^{-4}\EW{G_1^2(2,3)}&\nconv 0\label{23}\tag{C4}
\end{align}
Then $\frac{\CUn}{\sqrt{\V {U_n}}}$ and $\frac{\binom{n}{2}\mathcal{U}_n}{n\theta_n}$
converge in distribution to a standard normal random variable.
\end{theorem}

\section{A central limit theorem for a related U-statistics}
To make use of Theorem \ref{theo2} let us introduce the following variant of $\CUn$:
\begin{equation}
V_n=\sum\limits_{1\leq i<j\leq n+1}Z_{i,j}h_n(X_{n,i},X_{n,j}).
\label{eq:}
\end{equation}
Here, for fixed $n$, the $X_{n,i}$ are i.i.d. $\mathrm{Bin}(n-1,p)$-random variables (and $p=p_n$ is the same as above). The
$Z_{i,j}$ are i.i.d. $\mathrm{Ber}(p)$-distributed random variables independent of the $(X_{n,i})$, however their parameter $p$ is the same as for the $X_{n,i}$.
Finally,
\[h_n(x,y)=\frac{1}{(x+1)(y+1)}-\mu_n^2.\]
Note that, $h_n$ is symmetric and that
$\EW{h_n(X_{n,i},X_{n,j})}=0.$
It is  easy to check that \eqref{eq:hmoments} holds. 
Of course, we wish to apply the CLT in Theorem \ref{theo2} to prove a CLT for $V_n$.
In our setting the normalizing factor $n \theta_n$ can be directly computed (cf.\ Lemma \ref{lem:gammabeta} below).

In order to prove a CLT for $V_n$ we will check that \eqref{6}--\eqref{23} are fulfilled. We will prepare this by giving an alternative formula for the variables in our situation and by computing some inverse moments of the binomial distribution. This is done in the following lemmas:

\begin{lem}
	In the situation of this section we have
	\begin{align}
	\Phi(i,j)&=Z_{i,j}\cdot\Bigl(\frac{1}{(X_i+1)(X_j+1)}-\mu^2\Bigr)\notag\\
	\Psi_j(i)&=Z_{i,j}\cdot \mu\Bigl(\frac{1}{X_i+1}-\mu\Bigr)\notag\\
	\tilde{\Phi}(i,j)
	&=Z_{ij}\Big(\frac{1}{X_i+1}-\mu\Big)\Big(\frac{1}{X_j+1}-\mu\Big)\label{eq:tildephirepr}
	\end{align}
\end{lem}
\begin{proof}
	The expression for $\Phi$ immediately emerges from inserting $h$ into its definition.
	Furthermore, by measurability  and independence
	\begin{align*}
	\Psi_j(i)&=\mathbb{E}\Bigl[{Z_{i,j}\cdot\Bigl(\frac{1}{(X_i+1)(X_j+1)}-\mu^2\Bigr)\Big| X_i,Z_{i,j}}\Bigr]\\
	&=Z_{i,j}\cdot\Bigl(\frac{1}{X_i+1}\mathbb{E}\Bigl[\frac{1}{X_j+1}\Bigr]-\mu^2\Bigr)=Z_{i,j}\cdot\mu\Bigl(\frac{1}{X_i+1}-\mu\Bigr)\\
	\intertext{and}
	\tilde{\Phi}(i,j)&=Z_{ij}\Big(\frac{1}{X_i+1}\frac{1}{X_j+1}-\mu^2\Big)-Z_{ij}\mu\Big(\frac{1}{X_i+1}-\mu\Big)-Z_{ij}\mu\Big(\frac{1}{X_j+1}-\mu\Big)\notag\\
	&=Z_{ij}\Big(\frac{1}{X_i+1}-\mu\Big)\Big(\frac{1}{X_j+1}-\mu\Big)
	\end{align*}
\end{proof}
For what follows let
$$
\sigma^2\coloneqq\sigma_n^2\coloneqq\mathbb{E}\Bigl[\Bigl(\frac{1}{X_1+1}\Bigr)^2\Bigr].
$$
\begin{lem}\label{lem:gammabeta}
In the situation of this section we have
$\gamma_n^2=p\mu^2(\sigma^2-\mu^2)$ and $\beta_n^2=p(\sigma^4-\mu^4)$.
\end{lem}
\begin{proof}
Using the independence of the random variables straightforward calculations give
\begin{align*}
	\gamma_n^2&=\E\bigl[Z_{1,2}^2\mu^2\bigl(\frac{1}{X_1+1}-\mu\bigr)^2\bigr]
=p\mu^2\left(\sigma^2-2\mu^2+\mu^2\right)=p\mu^2\left(\sigma^2-\mu^2\right),
\end{align*}
as well as
\begin{align*}
\beta_n^2&=\E\bigl[Z_{1,2}^2\bigl(\frac{1}{X_1+1}\frac{1}{X_2+1}-\mu^2\bigr)^2\bigr]
=p(\sigma^4-2\mu^4+\mu^4)=p(\sigma^4-\mu^4).
\end{align*}
\end{proof}

In what follows, let $X\sim \mathrm{Bin}(n-1,p)$.
\begin{lem}\label{lem:firstmom}
For $\mu=\EW{\frac{1}{X+1}}$, it holds
$\mu^k=\frac{1}{(np)^k}\left(1-O(e^{-np})\right)=\Theta\bigl(\frac{1}{(np)^k}\bigr).$
\end{lem}

\begin{proof}
Note that
\begin{multline*}
\mu=\sum\limits_{k=0}^{n-1}\binom{n-1}{k}p^k(1-p)^{n-1-k}\cdot \frac{1}{k+1}=
\frac{1}{np}\sum\limits_{k=0}^{n-1}\binom{n}{k+1}p^{k+1}(1-p)^{n-(k+1)}
\\=\frac{1}{np}\sum\limits_{k=1}^{n}\binom{n}{k}p^{k}(1-p)^{n-k}=\frac{1}{np}(1-(1-p)^n)\leq\frac{1}{np}.
\end{multline*}

On the other hand, using $1+x\leq e^x$  we have $(1-p)^n\leq e^{-np}$ and thus
\begin{multline*}
(1-e^{-np})^k=1+\sum\limits_{k=1}^n\binom{n}{k}(-1)^k(e^{-np})^k\\\geq1-\sum\limits_{\substack{k=1\\k \text{ odd}}}^n\binom{n}{k}\cdot(e^{-np})^k
\geq1-e^{-np}\cdot\sum\limits_{\substack{k=1\\k \text{ odd}}}^n\binom{n}{k}\geq 1-L_ke^{-np}\end{multline*}
for a constant $L_k$. These two inequalities give for all $k\in\N$
\begin{equation}
\mu^k=\frac{1}{(np)^k}(1-(1-p)^n)^k\geq\frac{1}{(np)^k}(1-e^{-np})^k\geq\frac{1}{(np)^k}(1-L_ke^{-np})
\label{eq:qabsch}
\end{equation}
\end{proof}

Furthermore, from Remark \ref{rem:zni}
\begin{equation}
\sigma^2=\Theta\left(\frac{1}{(np)^2}\right).
\label{eq:sigmaorder}
\end{equation}

Thus by Corollary \ref{cor:invmoments} we obtain  estimates for $\beta_n^2$ and $\gamma_n^2$.
\begin{cor}\label{cor:betagamma}
We have
a) $\gamma_n^2=p\mu^2(\sigma^2-\mu^2)=\Theta\left(\frac{p}{(np)^5}\right)$ and

b) $\beta_n^2=p(\sigma^4-\mu^4)=\Theta\left(\frac{p}{(np)^5}\right)$
\end{cor}

\begin{proof}
By Corollary \ref{cor:invmoments}, $\sigma^2-\mu^2=\EW{\Bigl(\frac{1}{X+1}-\mu\Bigr)^2}=\Theta\Bigl(\frac{1}{(np)^3}\Bigr).$ Applying Lemma \ref{lem:firstmom} shows a).

For b) we apply the same steps and \eqref{eq:sigmaorder} to obtain
\begin{align*}\sigma^4-\mu^4&=\sigma^4-\sigma^2\mu^2+\sigma^2\mu^2-\mu^4=\sigma^2\bigl(\sigma^2-\mu^2\bigr)+\mu^2\bigl(\sigma^2-\mu^2\bigr)=\Theta\Bigl(\frac{1}{(np)^5}\Bigr).\end{align*}
This yields the claim.
\end{proof}

We now can formulate and prove a CLT for $V_n$
\begin{prop}\label{CLTVn}
In the setting of this section, $V_n/n \theta_n$ converges in distribution to a standard Gaussian random variable.
\end{prop}

\begin{proof}
We will need to check conditions \eqref{6}--\eqref{23}. For the sake of clarity, we will split the proof into four lemmas, each of them showing one of the conditions.

\begin{lem} \label{lem:cond6}
In the setting of this section assumption \eqref{6} holds, i.e.
\[\frac{1}{n\theta_n^2}\E\bigl[\Bigl(\sum\limits_{j=2}^{n+1}\Psi_j(1)\Bigr)^2\1_{\bigl\{\bigl|\sum\limits_{j=2}^{n+1}\Psi_j(1)\bigr|\geq\frac{\tau\theta_n n}{2}\bigr\}}\Bigr]\nconv 0\]
\end{lem}

\begin{proof}
Firstly, by multi-index notation and independence
\begin{align*}
\E\Bigl[\Bigl(\sum\limits_{j=2}^{n+1}\Psi_j(1)\Bigr)^4\Bigr]&=\sum\limits_{\substack{\alpha=(\alpha_2,\dots,\alpha_{n+1})\\|\alpha|=4}}\binom{4}{\alpha}\E\bigl[\prod\limits_{j=2}^{n+1}\Psi_j(1)^{\alpha_j}\bigr]\\
&=\sum\limits_{\substack{\alpha=(\alpha_2,\dots,\alpha_{n+1})\\|\alpha|=4}}\binom{4}{\alpha}\E\Bigl[\mu^{4}\Bigl(\frac{1}{X_1+1}-\mu\Bigr)^{4}\Bigr]\prod\limits_{j=2}^{n+1}\E\bigl[Z_{1,j}^{\alpha_j}\bigr]\\
&=\mu^{4}\E\Bigl[\Bigl(\frac{1}{X_1+1}-\mu\Bigr)^{4}\Bigr]\underbrace{\sum\limits_{\substack{\alpha=(\alpha_2,\dots,\alpha_{n+1})\\|\alpha|=4}}\binom{4}{\alpha}\prod\limits_{j=2}^{n+1}\EW{Z_{1,j}^{\alpha_j}}}_{\Lambda}
\end{align*}
Now consider $\prod_{j=2}^{n+1}\E Z_{1,j}^{\alpha_j}$. This is $p^k$ if and only if the number of $\alpha_j$ with $\alpha_j\neq 0$ is $k$. Since we are interested in asymptotic orders in $n$, we may neglect the binomial coefficient. For each $k=1, \ldots, 4$ there are $n\cdots (n-k+1)$ possibilities to choose the numbers multiplied with $p^k$ for the expectation. Thus for each $k$ the order of $\Lambda$ is $(np)^k$.
Thus, the entire sum is of order $(np)^4$ at most, which by Lemmas \ref{lem:firstmom} and Corollary \ref{cor:invmoments} gives
\begin{align}\label{eq:4mompsi}
\E\Bigl[\Bigl(\sum\limits_{j=2}^{n+1}\Psi_j(1)\Bigr)^4\Bigr]&=\mu^{4}\E\Bigl[\Bigl(\frac{1}{X_1+1}-\mu\Bigr)^{4}\Bigr]\cdot O\left((np)^4\right)
=
O\left((np)^{-6}\right).
\end{align}

Finally, apply Cauchy-Schwarz to the product in the expectation to obtain
\begin{align*}
\frac{1}{n\theta_n^2}\E\Bigl[\Bigl(\sum\limits_{j=2}^{n+1}\Psi_j(1)\Bigr)^2\1_{\bigl\{\bigl|\sum\limits_{j=2}^{n+1}\Psi_j(1)\bigr|\geq\frac{\eps\theta_n n}{2}\bigr\}}\Bigr]&\leq\frac{1}{n\theta_n^2}\sqrt{\E\Bigl[\Bigl(\sum_{j=2}^{n+1}\Psi_j(1)\Bigr)^4\Bigr]\P\Bigl(\bigl|\sum\limits_{j=2}^{n+1}\Psi_j(1)\bigr|\geq\frac{\eps\theta_n n}{2}\Bigr)}\\
& \hspace{-2cm}\leq\frac{1}{n\theta_n^2}\sqrt{\E\Bigl[\Bigl(\sum_{j=2}^{n+1}\Psi_j(1)\Bigr)^4\Bigr]16 \E\Bigl[\Bigl(\sum\limits_{j=2}^{n+1}\Psi_j(1)\Bigr)^4\Bigr]\eps^{-4}\theta_n^{-4} n^{-4}}\\
&\hspace{-2cm}=\frac{1}{n\theta_n^2}\E\Bigl[\Bigl(\sum_{j=2}^{n+1}\Psi_j(1)\Bigr)^4\Bigr]\frac{4}{\eps^2\theta_n^2 n^2}\
\end{align*}
where for the second inequality we applied Markov's inequality.
We now use \eqref{eq:4mompsi}, \eqref{defn}, and Corollary \ref{cor:betagamma} for $\gamma_n^2=\Theta\left(\frac{p}{(np)^5}\right)$ to see that
\begin{align*}
\frac{1}{n\theta_n^2}\E\Bigl[\Bigl(\sum\limits_{j=2}^{n+1}\Psi_j(1)\Bigr)^2\1_{\bigl\{\bigl|\sum\limits_{j=2}^{n+1}\Psi_j(1)\bigr|\geq\frac{\eps\theta_n n}{2}\bigr\}}\Bigr]&\le O\left(\frac{1}{n^3\theta_n^4(np)^6}\right)\leq O\left(\frac{1}{n^3\frac{(np)^2p^2}{(np)^{10}}(np)^6}\right)\\
&=O\left(\frac{1}{n}\right).
\end{align*}

\end{proof}

\begin{lem}\label{lem:20}
In the setting of this section assumption \eqref{20} holds.
\end{lem}

\begin{proof}
Let $\eps>0$. Then by Cauchy-Schwarz, and Markov's inequality
\begin{align*}
\theta_n^{-2}\EW{\tilde{\Phi}^2(1,2)\1_{\{|\tilde{\Phi}(1,2)|\geq \eps\theta_n n\}}}
&=\theta_n^{-2}\sqrt{\EW{\tilde{\Phi}^4(1,2)}\WK{|\tilde{\Phi}(1,2)|\geq \eps\theta_n n}}\\
&\le\theta_n^{-2}\sqrt{\EW{\tilde{\Phi}^4(1,2)}\EW{\tilde{\Phi}^2(1,2)}\eps^{-2}\theta_n^{-2} n^{-2}}
\end{align*}
And by \eqref{eq:tildephirepr}, independence and $Z_{12}=Z_{12}^4$
\begin{align*}
\theta_n^{-2}\EW{\tilde{\Phi}^2(1,2)\1_{\{|\tilde{\Phi}(1,2)|\geq \eps\theta_n n\}}}
&\le\theta_n^{-2}\biggl(\E\Bigl[Z_{12}\Bigl(\frac{1}{X_1+1}-\mu\Bigr)^4\Bigr]\E\Bigl[\Bigl(\frac{1}{X_2+1}-\mu\Bigr)^4\Bigr]\frac{\EW{\tilde{\Phi}^2(1,2)}}{ \eps^2\theta_n^2 n^2}\biggr)^{1/2}
\end{align*}
Applying Corollary \ref{cor:invmoments}, Lemma \ref{lem:secondmomtildephi} and
$\theta_n^2\geq \beta_n^2\geq\beta_n^2-2\gamma_n^2$ (which follows from \eqref{defn}) together with
$\theta_n^2\geq np\gamma_n^2=\Theta\left(\frac{np^2}{(np)^5}\right)$ (which we get from Corollary \ref{cor:betagamma}) we arrive at
\begin{align*}
\theta_n^{-2}\EW{\tilde{\Phi}^2(1,2)\1_{\{|\tilde{\Phi}(1,2)|\geq \eps\theta_n n\}}}& \le \theta_n^{-2}\sqrt{p\cdot O\left(\frac{1}{(np)^6}\right)^2\frac{\beta_n^2-2\gamma_n^2}{ \eps^2\theta_n^2 n^2}}\\
&\leq\Theta\Bigl(\frac{(np)^5}{np^2}\Bigr)O\left(\frac{\sqrt{p}}{(np)^6}\right)\sqrt{\frac{\theta_n^2}{ \eps^2\theta_n^2 n^2}}
=O\left(\frac{\sqrt{p}}{(np)^3}\right),
\end{align*}
which ensures that the term converges to 0.
\end{proof}

The last two assumptions can be proved in a similar fashion.

\begin{lem}\label{prop:21}
In the situation of this section assumption \eqref{21} holds.
\end{lem}

\begin{proof}
Let $\eps>0$. Again Cauchy-Schwarz' and Markov's inequality
\begin{align*}
\theta_n^{-2}p\E\Bigl[\tilde{H}(1,1)\1_{\{|\tilde{H}(1,1)|\geq \frac{\eps\theta_n^2 n}{p}\}}\Bigr]\hspace{-2cm}&\hspace{2cm}\leq\theta_n^{-2}p\sqrt{\E[\tilde{H}^2(1,1)]p^2 \eps^{-2}\theta_n^{-4} n^{-2} \EW{\tilde{H}^2(1,1)}}\\
&=\theta_n^{-2}\EW{\tilde{H}^2(1,1)}\frac{p^2}{\eps\theta_n^2 n}=\theta_n^{-2}\EW{\EW{\tilde{h}^2(1,2)\mid X_1}^2}\frac{p^2}{\eps\theta_n^2 n}
\end{align*}
where we applied \eqref{eq:tildeG}. By Jensen's inequality, \eqref{eq:tildephirepr}, and independence we get
\begin{align*}
\theta_n^{-2}p\EW{\tilde{H}(1,1)\1_{\left\{|\tilde{H}(1,1)|\geq \frac{\eps\theta_n^2 n}{p}\right\}}}
&\leq\theta_n^{-2}\EW{\tilde{h}^4(1,2)}\frac{p^2}{\eps\theta_n^2 n}\\
&=\theta_n^{-2}\E\Bigl[\Bigl(\frac{1}{X_1+1}-\mu\Bigr)^4\Bigr]\E\Bigl[\Bigl(\frac{1}{X_2+1}-\mu\Bigr)^4\Bigr]\frac{p^2}{\eps\theta_n^2 n}
\end{align*}
Finally, apply Corollaries \ref{cor:invmoments} and \ref{cor:betagamma} to find
$\theta_n^2\geq np\gamma_n^2=\Theta\left(\frac{np^2}{(np)^5}\right)$ and thus
\begin{align*}
\theta_n^{-2}p\EW{\tilde{H}(1,1)\1_{\left\{|\tilde{H}(1,1)|\geq \frac{\eps\theta_n^2 n}{p}\right\}}}&\le O\left(\frac{1}{(np)^6}\right)^2\frac{p^2}{\eps\theta_n^4 n}
=O\left(\frac{1}{(np)^6}\right)^2\Theta\left(\frac{1}{\frac{np^2}{(np)^5}}\right)^2 \frac{p^2}{\eps n}\\
=O\left(\frac{1}{n(np)^4}\right)
\end{align*}
which immediately ensures that the term converges to 0.
\end{proof}

\begin{lem}\label{lem:cond7}
Under the assumptions of this section Condition \eqref{23} holds.
\end{lem}

\begin{proof}
We consider the fourth moment of $h(i,j)$ for  $i\neq j$. By Jensen and independence between $X_i$ and $X_j$
\begin{align*}
\EW{h^4(i,j)}
&=\E\Bigl[\Bigl(\frac{1}{(X_i+1)(X_j+1)}\Bigr)^4\Bigr]-4\mu^2\E\Bigl[\Bigl(\frac{1}{(X_i+1)(X_j+1)}\Bigr)^3\Bigr]\\
&\phantom{==}+6\mu^4\E\Bigl[\Bigl(\frac{1}{(X_i+1)(X_j+1)}\Bigr)^2\Bigr]-4\mu^6\E\Bigl[\frac{1}{(X_i+1)(X_j+1)}\Bigr]+\mu^8\\
&\leq\E\Bigl[\Bigl(\frac{1}{X_i+1}\Bigr)^4\Bigr]^2+6\E\Bigl[\Bigl(\frac{1}{X_i+1}\Bigr)^4\Bigr]^2-7\mu^8\\
&=7\E\Bigl[\Bigl(\frac{1}{X_i+1}\Bigr)^4\Bigr]^2-7\mu^8
\end{align*}

We apply Remark \ref{rem:zni} and Lemma \ref{lem:firstmom}:
\begin{align*}
\EW{h^4(i,j)}&\leq7\left(\frac{1}{(np)^4}\Big(1+\Theta\bigl(\frac{1}{np}\bigr)\Big)\right)^2-7\frac{1}{(np)^8}\Bigl(1-\frac{1}{np}\Bigr)\\
&\leq7\frac{1}{(np)^8}\left(1+\Theta\Bigl(\frac{1}{np}\Bigr)\right)-7\frac{1}{(np)^8}\Bigl(1-\frac{1}{np}\Bigr)=O\left(\frac{1}{(np)^9}\right)\label{eq:phifourth}
\end{align*}
The rest of the proof is completed analogously to the one of Lemma \ref{lem:cond6}. By Cauchy-Schwarz and Jensen as well as the definition of $h$
\begin{align*}
\theta_n^{-4}\EW{G_1^2(2,3)}
&\le\theta_n^{-4}\EW{\EW{\Phi^2(1,2)\Phi^2(1,3)\mid X_2,X_3,Z_{1,2},Z_{1,3}}}\\
&=\theta_n^{-4}\EW{Z_{1,2},Z_{1,3}\EW{h^2(1,2)h^2(1,3)\mid X_2,X_3}}\\
&=\theta_n^{-4}\EW{Z_{1,2},Z_{1,3}}\EW{h^2(1,2)h^2(1,3)}
\end{align*}
With another application of Cauchy-Schwarz and finally, Corollary \ref{cor:betagamma}
\begin{align*}
\theta_n^{-4}\EW{G_1^2(2,3)}&\le \theta_n^{-4}p^2\sqrt{\EW{h^4(1,2)}\EW{h^4(1,3)}}=\theta_n^{-4}p^2\EW{h^4(1,2)}\\
&\leq O\left(\frac{(np)^{10}}{(np^2)^2}\right)p^2O\left(\frac{1}{(np)^9}\right)
= O\left(\frac{1}{np}\right)
\end{align*}
Since $np\to \infty$ this proves the claim.
\end{proof}

The proposition is now proved by combining the four lemmas above.
\end{proof}

\section{Removing the random centering}
While Proposition \ref{CLTVn} is quite nice in itself, a problem in the  application is that we are not interested in $Z_{ij}\frac{1}{(X_i+1)(X_j+1)}-Z_{ij}\mu^2$ but rather in
$Z_{ij}\frac{1}{(X_i+1)(X_j+1)}-p\mu^2$.
This does not seem like a big step, however it requires some work. Note that
by the classical CLT for sums of independent random variables
\begin{equation}
\frac{1}{\sqrt{\binom{n+1}{2}}p(1-p)}\sum\limits_{i<j}(Z_{ij}-p)\dconv \mathcal{N}(0,1),
\label{eq:classical}
\end{equation}
However, the factor in front of the sum is not quite what we want yet: 
\begin{lem}\label{lem:limitpconv}
We have that
\[\mu^2\frac{\sqrt{\binom{n+1}{2}}p(1-p)}{n\theta_n}\nconv \sqrt{\frac{p^*(1-p^*)}{2}}.\]
\end{lem}
\begin{proof}
By Corollary \ref{cor:betagamma} and $np\to\infty$, we know that $\theta_n^2\approx np\gamma_n^2$ 
as $\beta_n^2/2$ is negligible compared to $np\gamma_n^2$. Furthermore, by Lemma \ref{lem:gammabeta} we have that $\gamma_n^2=p\mu^2(\sigma^2-\mu^2)$. Consider $\sigma^2-\mu^2$: By Remark \ref{rem:zni} and Lemma \ref{lem:firstmom}
\begin{align*}
\sigma^2-\mu^2&=\E\Bigl[\Bigl(\frac{1}{X+1}\Bigr)^2\Bigr]-\frac{1}{(np)^2}\left(1-(1-p)^n\right)^2\\
&=\frac{1}{(np)^2}\left(\frac{1-p}{np}+\Theta\Bigl(\frac{1}{(np)^2}\Bigr)\right)+\frac{1}{(np)^2}\bigl((1-p)^{2n}-2(1-p)^n\bigr)\\
&\approx\frac{1-p}{(np)^3}
\end{align*}
From Lemma \ref{lem:firstmom} we also find $\mu\approx\frac{1}{np}$. 
Putting everything together:
\begin{align*}
\mu^2\frac{\sqrt{\binom{n+1}{2}}p(1-p)}{n\theta_n}&\approx\frac{\mu^2\frac{n}{\sqrt{2}}p(1-p)}{n\sqrt{np^2\mu^2(\sigma^2-\mu^2)}}
&\approx\frac{1}{np}\frac{(1-p)}{\sqrt{2n\frac{1-p}{(np)^3}}}\approx\frac{(1-p)(np)^{1/2}}{\sqrt{2n(1-p)}}=\sqrt{\frac{(1-p)p}{2}}
\end{align*}
proving the claim.
\end{proof}

Define
\begin{equation}
Z_n\coloneqq\frac{1}{n\theta_n}\sum\limits_{i<j}\mu^2(Z_{ij}-p).
\label{eq:defzn}
\end{equation}
Combining \eqref{eq:classical} and Lemma \ref{lem:limitpconv}, and letting $\mathfrak{N}$ denote a standard normal random variable,
we see that
\begin{equation}
Z_n
\dconv \sqrt{\frac{(1-p^*)p^*}{2}}\cdot\mathfrak{N},
\label{eq:resttermconv1}
\end{equation}
if $p^*\notin \{0,1\}$.  
If $p^*\in\{0,1\}$, we obtain that
$Z_n\dconv 0$ and therefore
$Z_n\Pconv 0.$
Thus, if $p^*\in\{0,1\}$, we see that
\begin{equation}
V_n+Z_n\underset{p^*\in\{0,1\}}{\dconv} \mathfrak{N}.
\label{eq:pconv0case}
\end{equation}

For the other cases, we follow the approach of the proof of Theorem 1 in \cite{Janson84}.
Let $V\sim\mathcal{N}\left(0,1\right)$ and $W\sim\mathcal{N}\left(0,\frac{(1-p^*)p^*}{2}\right)$ be independent random variables with characteristic functions $\psi_V(t)$ and $\psi_W(t)$, resp.
Set
\[\varphi(s,t)=\psi_V(s)\psi_W(t).\]
For almost every realization $z=\left(z_{i,j}\right)_{1\leq i<j\leq n+1}$ of $Z=\left(Z_{i,j}\right)_{1\leq i<j\leq n+1}$ we find that
\begin{lem}\label{eq:jointconvcond1}
\begin{equation*}
\EW{e^{isV_n}\mid Z=z}\nconv e^{-\frac{1}{2}s^2}=\psi_V(s).
\end{equation*}
\end{lem}

\begin{proof}
By the strong law of large numbers, we have that
\begin{equation}
\sum\limits_{i<j}Z_{ij}= \binom{n}{2}p(1+o(1)) \quad \mbox{ a.a.s.}
\label{eq:slln1}
\end{equation}
 in the sense that, for almost every sequence $z$ there is a $o(1)$-term such that \eqref{eq:slln1} is fulfilled.
By squaring this equation
\begin{equation}\label{squared}
\sum\limits_{i<j}Z_{ij}\sum\limits_{k<l}Z_{kl}= \left(\binom{n}{2}\right)^2p^2(1+o(1)) \quad \mbox{ a.a.s.}
\end{equation}
Note that those pairs $Z_{ij}Z_{kl}$ in \eqref{squared} where at least two of $i,j,k,l$ are equal, are asymptotically negligible, because $|Z_{ij}|\le 1$ and there are at most $O(n^3)$ of such pairs, while the total number of summands is $\left(\binom n 2\right)^2$.
Considering
$Z_{ij}Z_{kl}$ with $i<j$ and $k<l$ pairwise different we realize that there 6 different ways to order the indices (one of them being $i<j<k<l$) and all of them occur equally often.
Hence
\begin{equation}
\sum\limits_{i<j<k<l}Z_{ij}Z_{kl}= \binom{n}{4}p^2(1+o(1))=\frac{1}{6}\binom{n}{2}\binom{n-2}{2}p^2(1+o(1))  \quad \mbox{ a.a.s.}
\label{eq:slln2}
\end{equation}
%
We follow the concept of the proof of Theorem 2 in \cite{Janson84}.
Let
\[T_n=\frac{1}{n\alpha_n}\sum\limits_{i<j}h(i,j)\quad \text{ and } \quad V_n=\frac{1}{n\theta_n}\sum\limits_{i<j}Z_{i,j}h(i,j).\]
with $\alpha_n^2:=n \tilde \gamma_n^2+\tilde \beta_n^2/2$, $\alpha_n \ge 0$ chosen as in \cite{Mal87} and $\tilde{\beta}_n^2:= \beta_n^2/p$ and $\tilde{\gamma}_n^2:=\gamma_n^2/p$.

Then, 
\begin{align*}\E\left[T_n^2\right]&=\frac{1}{n^2\alpha_n^2}\sum\limits_{i<j}\sum\limits_{k<l}\EW{h(i,j)h(k,l)}
\eqqcolon\frac{1}{n^2\alpha_n^2}\sum\limits_{c=0}^2A_cB_c\end{align*}
where
\[A_c=\underset{|\{i,j\}\cap\{k,l\}|=c}{\sum\limits_{i<j}\sum\limits_{k<l}}1=\begin{cases}
\binom{n}{2}\binom{n-2}{2},&c=0\\
\binom{n}{2}(2n-4),&c=1\\
\binom{n}{2},&c=2
\end{cases}\]
and \[B_c=\begin{cases}
\EW{h(1,2)h(3,4)}=0,&c=0\\
\EW{h(1,2)h(1,3)}=\tilde{\gamma}_n^2,&c=1\\
\EW{h(1,2)^2}=\tilde{\beta}_n^2,&c=2
\end{cases},\]
so that
\begin{equation}
\E\left[T_n^2\right]\approx\frac{1}{n^2\alpha_n^2}\binom{n}{2}\left(2n\tilde{\gamma}_n^2+\tilde{\beta}_n^2\right)=1+o(1)
\label{eq:vun}
\end{equation}
The conditional second moment of $V_n$ given $Z=z$ can be calculated in the same fashion: 
\begin{align*}
\EW{V_n^2|Z=z}&=\frac{1}{n^2\theta_n^2}\sum\limits_{i<j}\sum\limits_{k<l}z_{i,j}z_{k,l}\EW{h(i,j)h(k,l)}\eqqcolon\frac{1}{n^2\theta_n^2}\sum\limits_{c=0}^2A_c'B_c
\end{align*}
where
$A_c'=\underset{|\{i,j\}\cap\{k,l\}|=c}{\sum\limits_{i<j}\sum\limits_{k<l}}z_{i,j}z_{k,l}$.

From \eqref{squared}, $$A_0'+A_1'+A_2'=\sum\limits_{i<j}\sum\limits_{k<l}z_{i,j}z_{k,l}=\binom{n}{2}^2p^2(1+o(1))$$
as well as $A_0'=\binom{n}{2}\binom{n-2}{2}p^2(1+o(1))$ from \eqref{eq:slln2} and $A_2'=\binom{n}{2}p(1+o(1))$ from \eqref{eq:slln1}
for almost every sequence $z$.
We can therefore calculate
\begin{align*}
A_1'&=(A_0'+A_1'+A_2')-A_0'-A_2'=(1+o(1))\Big(\binom{n}{2}^2p^2-\binom{n}{2}\binom{n-2}{2}p^2-\binom{n}{2}p\Big)\\
&=(1+o(1))\Big(\binom{n}{2}p(2np-3p-1)\Big),
\end{align*}
thus $A_1=\binom{n}{2}p\cdot2np(1+o(1))$ for almost every sequence $z$.

We obtain
\begin{align*}
\EW{V_n^2|Z=z}&=\frac{1}{n^2\theta_n^2}\binom{n}{2}\left(2np^2\tilde{\gamma}_n^2+\tilde{\beta}_n^2p\right)(1+o(1))=\frac{1}{n^2\theta_n^2}\binom{n}{2}\left(2np\gamma_n^2+\beta_n^2\right)(1+o(1))\\
&=1+o(1)
\end{align*}
for almost every sequence $z$.
Finally, similar to $\EW{V_n^2\mid Z=z}$ one can calculate
\[\EW{V_nT_n| Z=z}=\frac{1}{n^2\alpha_n\theta_n}\binom{n}{2}\left(2np\tilde{\gamma}_n^2+\tilde{\beta}_n^2p\right)(1+o(1))=\frac{1}{n^2\alpha_n\theta_n}\binom{n}{2}2p\alpha_n^2(1+o(1))\]
for almost every sequence $z$. Observing that $\frac{p\alpha_n}{\theta_n}\approx\frac{p\sqrt{n\tilde{\gamma}_n}}{\sqrt{np\gamma_n}}=1$ from Corollary \ref{cor:betagamma} and $np\to\infty$ yields
\[\EW{V_nT_n| Z=z}=1+o(1)\]
for almost every sequence $z$ and therefore, combining above estimations
$$\EW{(T_n-V_n)^2|Z=z}=\EW{T_n^2}-2\EW{V_nT_n| Z=z+}\EW{V_n^2|Z=z}=o(1)$$
for almost all $z$.
Therefore, for almost all $z$ we obtain $|V_n^z-T_n|\Pconv 0$, where $V_n^z$ is $V_n$ restriced to $Z=z$. 

Applying the result from \cite{Mal87}, we know $T_n\dconv \mathfrak{N}$ (the conditions can be shown  analogously to  Lemma \ref{lem:cond6} -- Lemma \ref{lem:cond7}).
Thus  $V_n^z\dconv \mathfrak{N}$ for almost all $z$ and the claim holds.
\end{proof}

Now,
\begin{align*}
\EW{e^{isV_n+itZ_n}}-\varphi(s,t)&=\EW{e^{isV_n+itZ_n}}-e^{-\frac{1}{2}s^2}\psi_W(t)\\
&\hspace{-2cm}=\EW{e^{itZ_n}\left(\EW{e^{isV_n}\mid Z}-e^{-\frac{1}{2}s^2}\right)}+\left(\EW{e^{itZ_n}}-\psi_W(t)\right)e^{-\frac{1}{2}s^2},
\end{align*}
where both summands on the right hand side almost surely converge to 0 by Lemma \ref{eq:jointconvcond1} and \eqref{eq:resttermconv1}.
In total, we find that
$(V_n,Z_n)\dconv\left(V,W\right).$
By the continuous mapping theorem, we thus obtain
\begin{equation}
V_n+Z_n\dconv V+W=\mathcal{N}\left(0,1\right)+\mathcal{N}\left(0,\frac{(1-p^*)p^*}{2}\right)=\mathcal{N}\left(0,\,1+\frac{(1-p^*)p^*}{2}\right).
\label{eq:pconvnot0case}
\end{equation}

Including both cases \eqref{eq:pconv0case} and \eqref{eq:pconvnot0case} in one theorem we have seen that
\begin{theorem}\label{thm:main}
With the notations from before,
\begin{multline*}V_n+Z_n=\frac{1}{n\theta_n}\sum\limits_{i<j}\frac{Z_{ij}}{(X_i+1)(X_j+1)}-\frac{1}{n\theta_n}\binom{n+1}{2}\mu^2p\\=\frac{1}{n\theta_n}\sum\limits_{i<j}\left(\frac{Z_{ij}}{(X_i+1)(X_j+1)}-\mu^2p\right).\end{multline*}
Then the following holds:
$V_n+Z_n\dconv\mathfrak{V}$
where
$\mathfrak{V}=\mathcal{N}\left(0,1+\frac{(1-p^*)p^*}{2}\right)$.
\end{theorem}
\section{A CLT for $U_n$}
We now want to prove a CLT for the U-statistic $U_n$ defined in \eqref{eq:mainustat}.
In order to get the same result for $U_n$ as for $V_n+Z_n$, we want to show that the moments of $U_n$ behave like those of $V_n+Z_n$. If we also manage to show that the moments of $V_n+Z_n$ converge to those of its limiting distribution, this convergence also holds for $U_n$ and by the method of moments, the desired result holds. In order to show that the moments of $V_n+Z_n$ even converge to those of limiting distribution $\mathcal{N}(0,1+\frac{(1-p^*)p^*}{2})$, it suffices to show that $\sup_n\EW{(V_n+Z_n)^{2k}}<\infty$ for all $k$, by \cite{DasGupta2008}, Theorem 6.2.
We start with a combinatorial lemma which we will need later.

\begin{lem}\label{lem:combinatorial2}
For natural numbers $k_{ij}$ we have
$$\sum\limits_{\sum\limits_{i<j}k_{ij}=2k}p^{\sum\limits_{i<j}\1_{\{k_{ij}>0\}}}\prod\limits_{i=1}^{n+1}\1_{\left\{\sum\limits_{j=i+1}^{n}k_{ij}+\sum\limits_{j=1}^{i-1}k_{ji}\neq 1\right\}}= O\left((np)^kn^k\right).$$
\end{lem}

The proof of this lemma will be given in the appendix.

Next we will see
\begin{prop}\label{prop:limitedvn}
For any fixed $k\in\N$ we have
\[\EW{(V_n)^{2k}}= O(1).\]
\end{prop}

\begin{proof}
Since $V_n=\frac{1}{n\theta_n}\sum\limits_{i=1}^{n+1}\sum\limits_{j=1}^{i-1}\tilde{\Phi}(i,j)+\frac{1}{n\theta_n}\sum\limits_{i=1}^{n+1}\sum\limits_{\substack{j=1 \\j\neq i}}^{n+1}\Psi_j(i)$ it suffices to prove that
\[\E\Big[\big(\frac{1}{n\theta_n}\sum\limits_{i=1}^{n+1}\sum\limits_{j=1}^{i-1}\tilde{\Phi}(i,j)\big)^{2k}\Bigr]= O(1) \mbox{  and }
\E\Big[\big(\frac{1}{n\theta_n}\sum\limits_{i=1}^{n+1}\sum\limits_{\substack{j=1 \\j\neq i}}^{n+1}\Psi_j(i)\big)^{2k}\Big]= O(1)\]
Now,
\begin{align*}
\E\Big[\big((\frac{1}{n\theta_n}\sum\limits_{i<j}\tilde{\Phi}(i,j)\big)^{2k}\Big]\hspace{-4cm}&\hspace{4cm}=
\frac{1}{(n\theta_n)^{2k}}\sum\limits_{\sum\limits_{i<j}k_{ij}=2k}\binom{2k}{k_{1,1},\dots,k_{n,n+1}}
\E\Big[\prod\limits_{i<j}\Big(\tilde{\Phi}(i,j)\Big)^{k_{ij}}\Big]\\
&\le\frac{C_k}{(n\theta_n)^{2k}}\sum\limits_{\sum\limits_{i<j}k_{ij}=2k}
\E\Big[\prod\limits_{i<j}Z_{i,j}^{k_{ij}}\Big]
\E\Big[\prod\limits_{i<j}\Big(\frac{1}{X_i+1}-\mu\Big)^{k_{ij}}\Big(\frac{1}{X_j+1}-\mu\Big)^{k_{ij}}\Big]\\
&=C_k\frac{1}{(n\theta_n)^{2k}}\sum\limits_{\sum\limits_{i<j}k_{ij}=2k}
\E\Big[\prod\limits_{i<j}Z_{i,j}^{k_{ij}}\Big]
\E\Big[\prod\limits_{i=1}^{n+1}\Big(\frac{1}{X_i+1}-\mu\Big)^{\sum\limits_{j=i+1}^{n}k_{ij}+\sum\limits_{j=1}^{i-1}k_{ji}}\Big]\\
&=C_k\frac{1}{(n\theta_n)^{2k}}\sum\limits_{\sum\limits_{i<j}k_{ij}=2k}\prod\limits_{i<j}\E\Big[Z_{i,j}^{k_{ij}}\Big]
\prod\limits_{i=1}^{n+1}
\E\Big[\Big(\frac{1}{X_i+1}-\mu\Big)^{\sum\limits_{j=i+1}^{n}k_{ij}+\sum\limits_{j=1}^{i-1}k_{ji}}\Big]\\
\end{align*}
by reordering and independence.
From Corollary \ref{cor:invmoments} we get an upper bound for any moment of $\frac{1}{X_i+1}-\mu$. Additionally, if $k_{ij}>0$, $Z_{i,j}^{k_{ij}}=Z_{i,j}$, so that the expectation is 1, if $k_{ij}=0$ and $p$ otherwise. Thus
\begin{align*}
&\E\Big[\big((\frac{1}{n\theta_n}\sum\limits_{i<j}\tilde{\Phi}(i,j)\big)^{2k}\Big]\\
&= C_k\frac{1}{(n\theta_n)^{2k}}\sum\limits_{\sum\limits_{i<j}k_{ij}=2k}
p^{\sum\limits_{i<j}\1_{\{k_{ij}>0\}}}O\Big((np)^{-\sum\limits_{i=1}^{n+1}(\sum\limits_{j=i+1}^{n}k_{ij}+\sum\limits_{j=1}^{i-1}k_{ji})}\Big)
\prod\limits_{i=1}^{n+1}\1_{\big\{\sum\limits_{j=i+1}^{n}k_{ij}+\sum\limits_{j=1}^{i-1}k_{ji}\neq 1\big\}}\\
&= \frac{1}{(n\theta_n)^{2k}}\sum\limits_{\sum\limits_{i<j}k_{ij}=2k}
p^{\sum\limits_{i<j}\1_{\{k_{ij}>0\}}}O\big((np)^{-4k}\big)\prod\limits_{i=1}^{n+1}
\1_{\{\sum\limits_{j=i+1}^{n}k_{ij}+\sum\limits_{j=1}^{i-1}k_{ji}\neq 1\}}
\end{align*}
where we swallowed $C_k$ into the $O$-term.
Applying Lemma \ref{lem:combinatorial2} gives
\begin{align*}
\E\Big[\big((\frac{1}{n\theta_n}\sum\limits_{i<j}\tilde{\Phi}(i,j)\big)^{2k}\Big]
&\leq O\left(\frac{1}{(np)^{4k}}\right)\frac{1}{(n\theta_n)^{2k}}O\left(n^k(np)^k\right)\\
&= O\left(\frac{1}{(np)^{4k}}\right)O\left(\frac{(np)^{5k}}{n^{2k}(np^2)^k}\right)O\left(n^k(np)^k\right)
=O(1)
\end{align*}
For the second summand we estimate analogously
\begin{align*}
\E\Big[\big(\frac{1}{n\theta_n}\sum\limits_{i\neq j}\Psi_j(i)\big)^{2k}\Big]\hspace{-4cm}&\hspace{4cm}=\frac{1}{(n\theta_n)^{2k}}\sum\limits_{\sum\limits_{i\neq j}k_{ij}=2k}\binom{2k}{k_{1,1},\dots,k_{n,n+1}}\EW{\prod\limits_{i\neq j}\left(\Psi_j(i)\right)^{k_{ij}}}\\
&= O\left(\frac{(np)^{7k}}{n^k(np)^{6k}}\right)=O(p^k)= O(1)
\end{align*}
\end{proof}

\begin{prop}\label{prop:limitedzn}
For any fixed $k \in \N$ we have
\[\EW{(Z_n)^{2k}}= O(1)\]
\end{prop}

\begin{proof}
The proof is very similar to that of Proposition \ref{prop:limitedvn}. We will therefore not go into details.
\end{proof}

Now Proposition \ref{prop:limitedvn} and Proposition \ref{prop:limitedzn} together give
\[\EW{(V_n+Z_n)^{2k}}\leq 2^{2k}\left(\EW{V_n^{2k}}+\EW{Z_n^{2k}}\right)= O(1).\]
Then $\sup_n\EW{\left(V_n+Z_n\right)^{2k}}<\infty$. As a
as a consequence, by Theorem 6.2 in \cite{DasGupta2008}, we obtain
\begin{prop}
\begin{equation}
\EW{\left(V_n+Z_n\right)^k}\nconv \EW{\mathfrak{V}^k}
\label{eq:vnmomconv}
\end{equation}
for every fixed $k\in\N$,
where $\mathfrak{V}$ is a $\mathcal{N}\left(0,1+\frac{(1-p^*)p^*}{2}\right)$-distributed random variable.
\end{prop}

Next we show that the moments of $U_n$ have the same limits as those of $V_n+Z_n$.
\begin{prop}
We have
\[\lim\limits_{n\to\infty}\EW{U_n^{k}}=\lim\limits_{n\to\infty}\EW{\left(V_n+Z_n\right)^{k}}.\]
\end{prop}
\begin{proof}
Recall that by \eqref{eq:mainustat}
\[U_n=\frac{1}{n\theta_n}\sum\limits_{1\leq i<j\leq n+1}\left(\frac{a_{i,j}}{(\tilde{d}_i^j+1)(\tilde{d}_j^i+1)}-\mu^2p\right).\]
Then, for any fixed $k\in\N$, obviously
\begin{align}
&\EW{U_n^k}
=\frac{1}{(n\theta_n)^k}\E\Big[\sum\limits_{\sum\limits_{i<j}k_{ij}=k}\binom{k}{k_{1,2},\dots,k_{n,n+1}}
\prod\limits_{i<j}\Big(\frac{a_{ij}}{(\tilde{d}_i^j+1)(\tilde{d}_j^i+1)}-\mu^2p\Big)^{k_{ij}}\Bigr]\notag\\
&=\frac{1}{(n\theta_n)^k}\E\Big[\sum\limits_{\sum\limits_{i<j}k_{ij}=k}\binom{k}{k_{1,2},\dots,k_{n,n+1}}
\prod\limits_{i<j}\sum\limits_{l=0}^{k_{ij}}\Big(\frac{a_{ij}}{(\tilde{d}_i^j+1)(\tilde{d}_j^i+1)}\Big)^l
\Big(-\mu^2p\Big)^{k_{ij}-l}\Big]\notag\\
&=\frac{1}{(n\theta_n)^k}\sum\limits_{\sum\limits_{i<j}k_{ij}=k}\binom{k}{k_{1,2},\dots,k_{n,n+1}}
\sum\limits_{l_{1,2}=0}^{k_{1,2}}\dots\sum\limits_{l_{n,n+1}=0}^{k_{n,n+1}}
\E\Big[\prod\limits_{i<j}\big(\frac{a_{ij}}{(\tilde{d}_i^j+1)(\tilde{d}_j^i+1)}\big)^{l_{ij}}
\big(-\mu^2p\big)^{k_{ij}-l_{ij}}\Bigr]\label{eq:momentseq3}
\end{align}
Denote $E_1:=\E\Big[\prod\limits_{i<j}\big(\frac{a_{ij}}{(\tilde{d}_i^j+1)(\tilde{d}_j^i+1)}\big)^{l_{ij}}
\big(-\mu^2p\big)^{k_{ij}-l_{ij}}\Bigr]$.
Note that in the last expression the maximum number of pairs $(i,j)$ with $k_{ij}\neq 0$ is $k$. Let $1\leq m\leq k$ be the number of such pairs and let $(i_r,j_r)$ for $1\leq r\leq m$ be the corresponding indices. Thus,
{\small{\begin{equation}
E_1=
\E\Big[\prod\limits_{r=1}^{m}
\big(\frac{a_{i_r,j_r}}{(\tilde{d}_{i_r}^{j_r}+1)(\tilde{d}_{j_r}^{i_r}+1)}\big)^{l_{{i_r},{j_r}}}
\big(-\mu^2p\big)^{k_{{i_r},{j_r}}-l_{{i_r},{j_r}}}\Big]
\label{eq:momentseq1}
\end{equation}}}
Let $b_{i_r}
=\sum\limits_{\substack{k\in R,\\k\neq i_r,j_r}}a_{i_r,k}$,
$R=\{i_r,j_r:1\leq r\leq m\}$, and $\tilde{d}_{i_r}^R=\sum\limits_{k\notin R}a_{i_r,k}$. Hence, $\tilde{d}_{i_r}^R$ is the number of neighbors of $i_r$ without the other vertices in $R$ (which are finitely many). 
Therefore,
{\small{\[
E_1=
\E\Big[\prod\limits_{r=1}^{m}\Big(\frac{a_{i_r,j_r}}{(\tilde{d}_{i_r}^R+b_{i_r}+1)(\tilde{d}_{j_r}^R+b_{j_r}+1)}\Big)^{l_{{i_r},{j_r}}}
\big(-\mu^2p\big)^{k_{{i_r},{j_r}}-l_{{i_r},{j_r}}}\Big].
\]}}
Summing over all possible realizations of $a_{i',j'}$ for $i',j'\in R$, $i'\neq j'$ and denoting this sum by $\sum\limits_{A}$ (and the corresponding realization of $a_{i',j'}$ by $A_{i',j'}$), we may also replace the 
$b_{i_r}$-terms by their realizations $B_{i_r}$, as they are sums of some $a_{i_r,k}$, where $i_r,k\in R$, and obtain
\[E_2:=\E\Big[\sum\limits_{A}\1_{\Big\{\substack{a_{i',j'}=A_{i',j'}\\i', j'\in R, i'\neq j'}\Big\}}\prod\limits_{r=1}^{m}\Big(\frac{A_{i_r,j_r}}{(\tilde{d}_{i_r}^R+B_{i_r}+1)
(\tilde{d}_{j_r}^R+B_{j_r}+1)}\Big)^{l_{{i_r},{j_r}}}\Big(-\mu^2p\Big)^{k_{{i_r},{j_r}}-l_{{i_r},{j_r}}}\Bigr].\]
Note that $A_{i',j'}$ and $B_{i_r}$ are no longer random.
Moreover, either
the $\tilde{d}^R_{i_r}, \tilde{d}^R_{j_r}$ are independent of each other and of the $a_{i',j'}$ for $i',j'\in R$ and we can apply this (note that the summation of all possible realizations $A$ is over the entire term). Or
there are multiple identical indices among the $i_r,j_r$. Then we know from Proposition \ref{prop:asymptindep} that
\begin{equation}\label{eq:asymp100}
\EW{\Bigl(\frac{1}{\tilde{d}_{i_r}^R+B_{i_r}+1}\Bigr)^{a+b}}\approx\EW{\Bigl(\frac{1}{\tilde{d}_{i_r}^R+1}\Bigr)^{a}}\cdot\EW{\Bigl(\frac{1}{\tilde{d}_{i_r}^R+1}\Bigr)^{b}},
\end{equation}
such that we can still separate expectations asymptotically.
Altogether,
\begin{align*}
&E_2\approx\sum\limits_{A}\E\Big[\1_{\{\substack{a_{i',j'}=A_{i',j'}\\i', j'\in R, i'\neq j'}\}}\Big]\prod\limits_{r=1}^{m}\E\Big[\Big(\frac{1}{\tilde{d}_{i_r}^R+B_{i_r}+1}\Big)^{l_{{i_r},{j_r}}}\Big]
\prod\limits_{r=1}^{m}\E\Big[\Big(\frac{1}{\tilde{d}_{j_r}^R+B_{j_r}+1}\Big)^{l_{{i_r},{j_r}}}\Bigr]\\
&\hspace{6cm}\cdot\prod\limits_{r=1}^{m}A_{i_r,j_r}^{l_{i_r,j_r}}\prod\limits_{r=1}^{m}\Big(-\mu^2p\Big)^{k_{{i_r},{j_r}}-l_{{i_r},{j_r}}}
\end{align*}
By Proposition \ref{prop:asymptindep}, we know that
\[\E\Big[\Big(\frac{1}{\tilde{d}_{i_r}^R+B_{i_r}+1}\Big)^{l_{{i_r},{j_r}}}\Big]\approx \E\Big[\Big(\frac{1}{X_{i_r}+1}\Big)^{l_{{i_r},{j_r}}}\Bigr],\]
for any $\mathrm{Bin}(n-1,p)$-distributed random variable $X_{i_r}$, as we removed $|R|\leq 2m$ vertices, which is constant in $n$.
In particular, we can choose $X_i$, $i=1,\dots,n+1$ independent of each other.
Furthermore, the $a_{i,j}$ are all independent and identically distributed as $Z_{i,j}$.
We therefore obtain
$\sum\limits_{A}\E\Big[\1_{\big\{\substack{a_{i',j'}=A_{i',j'}\\i', j'\in R, i'\neq j'}\big\}}\Big]=\sum\limits_{A}\E\Big[\1_{\big\{\substack{Z_{i',j'}=A_{i',j'}\\i', j'\in R, i'\neq j'}\big\}}\Big]$.
Thus using independence, Proposition \ref{prop:asymptindep}, and \eqref{eq:asymp100} as above, we may rewrite the entire expectation \eqref{eq:momentseq1} as
\begin{align*}
&\sum\limits_{A}\E\Big[\1_{\big\{\substack{Z_{i',j'}=A_{i',j'}\\i', j'\in R, i'\neq j'}\big\}}\Big]\prod\limits_{r=1}^{m}\E\Big[\Big(\frac{1}{X_{i_r}+1}\Big)^{l_{{i_r},{j_r}}}\Big]
\prod\limits_{r=1}^{m}\E\Big[\Big(\frac{1}{X_{j_r}+1}\Big)^{l_{{i_r},{j_r}}}\Big]
\prod\limits_{r=1}^{m}A_{i_r,j_r}^{l_{i_r,j_r}}\prod\limits_{r=1}^{m}\left(-\mu^2p\right)^{k_{{i_r},{j_r}}-l_{{i_r},{j_r}}}\\
&\approx\sum\limits_{A}\EW{\1_{\left\{\substack{Z_{i',j'}=A_{i',j'}\\i', j'\in R, i'\neq j'}\right\}}\prod\limits_{r=1}^{m}\left(\frac{1}{X_{i_r}+1}\right)^{l_{{i_r},{j_r}}}\left(\frac{1}{X_{j_r}+1}\right)^{l_{{i_r},{j_r}}}Z_{i_r,j_r}^{l_{i_r,j_r}}\left(-\mu^2p\right)^{k_{{i_r},{j_r}}-l_{{i_r},{j_r}}}}\\
&=\E\Big[\prod\limits_{r=1}^{m}
\Big(\frac{1}{X_{i_r}+1}\Big)^{l_{{i_r},{j_r}}}\Big(\frac{1}{X_{j_r}+1}\Big)^{l_{{i_r},{j_r}}}
Z_{i_r,j_r}^{l_{i_r,j_r}}\big(-\mu^2p\big)^{k_{{i_r},{j_r}}-l_{{i_r},{j_r}}}
\Big]\\
&=\E\Big[\prod\limits_{i<j}\Big(\frac{Z_{i,j}}{(X_{i}+1)(X_{j}+1)}\Big)^{l_{{i},{j}}}\big(-\mu^2p\big)^{k_{{i},{j}}-l_{{i},{j}}}\Big].
\end{align*}
Now we reverse the steps we did to obtain \eqref{eq:momentseq3} and arrive at:
\begin{align*}
\E[U_n^k]
&\approx
\frac{1}{(n\theta_n)^k}\E\Big[\Big(\sum\limits_{i<j}\Big(\frac{Z_{i,j}}{(X_i+1)(X_j+1)}-\mu^2p\Big)\Big)^k\Big]
=\E\Big[\big(V_n+Z_n\big)^k\Big].
\end{align*}
Hence,
\[\lim\limits_{n\to\infty}\EW{U_n^k}=\lim\limits_{n\to\infty}\EW{V_n^k}.\]
This completes the proof.
\end{proof}

From \eqref{eq:vnmomconv} we know  $$\EW{U_n^k}\nconv \EW{\mathfrak{V}^k}.$$
Thus by the method of moments we have proved
\begin{theorem}
\label{thm:indepconv}
Under the above assumptions
\[U_n=\frac{1}{n\theta_n}\sum\limits_{i<j}\left(\frac{a_{ij}}{d_id_j}-\mu^2p\right)\dconv \mathfrak{V}.\]
\end{theorem}

\section{Proof of Theorem \ref{main_theo}}
Recall that by \eqref{eq:Hirepres}
we know that
$$H^i=(n-2)+\sum\limits_{k=1}^{n+1}\lambda_k^2+O\left(\sum\limits_{k=2}^{n+1}\lambda_k^3\right).$$
Moreover, in \eqref{Ustat1}, we saw that
$
\sum\limits_{k=1}^{n+1}\lambda_k^2=\sum\limits_{i,j=1}^{n+1}\frac{a_{ij}}{d_id_j}=2\sum\limits_{i<j}\frac{a_{ij}}{d_id_j}.
$
To this second term we can apply Theorem \ref{thm:indepconv} to obtain a CLT.
It thus remains to deal with the $O$-term in the above expression for $H^i$. Just as above we see that
$\sum\limits_{k=1}^{n+1}\lambda_k^3=\mathrm{tr}(B^3)=\sum\limits_{i=1}^{n+1}z_{ii},$
where $z_{ij}$ are the entries of $B^3$.
The entries of $B$ are given by $b_{ij}=\frac{a_{ij}}{\sqrt{d_i}\sqrt{d_j}}$, hence
\[z_{ii}=
\sum\limits_{j,k=1}^{n+1}\frac{a_{ij}a_{jk}a_{ki}}{d_id_jd_k}=\sum\limits_{\substack{i,j,k=1\\i,j,k \text{ p.d.}}}^{n+1}\frac{a_{ij}a_{jk}a_{ki}}{d_id_jd_k}\]
since $a_{ii}=0$ (where ''p.d.'' stands for ''pairwise different'').
Since $\lambda_1=1$,
\begin{equation}
\sum\limits_{k=2}^{n+1}\lambda_k^3=\sum\limits_{k=1}^{n+1}\lambda_k^3-1=\sum\limits_{\substack{i,j,k=1\\i,j,k \text{ p.d.}}}^{n+1}\frac{a_{ij}a_{jk}a_{ki}}{d_id_jd_k}-1
\label{eq:thirdmomenteq1}
\end{equation}
We want to prove that
\begin{prop}\label{prop:ohoch3term}
\[\frac{1}{n\theta_n}\sum\limits_{k=2}^{n+1}\lambda_k^3\Pconv 0.\]
\end{prop}
\begin{proof}
We consider the second moment of the quantity of interest:
\begin{align}
\E\Big[\Big(\sum\limits_{k=2}^{n+1}\lambda_k^3\Big)^2\Big]
=
\E\Big[\Big(\sum\limits_{i=1}^{n+1}z_{ii}\Big)^2\Big]-2\E\Big[\sum\limits_{k=1}^{n+1}z_{ii}\Big]+1.\label{propsplit1}
\end{align}
For the second summand we see:
\begin{align}
\E\Big[\sum\limits_{i=1}^{n+1}z_{ii}\Big]
&=\sum\limits_{\substack{i,j,k=1\\i,j,k\text{ p.d.}}}^{n+1}\E\Big[\frac{a_{ij}a_{jk}a_{ki}}{d_id_jd_k}\Big]
=\sum\limits_{\substack{i,j,k=1\\i,j,k\text{ p.d.}}}^{n+1}\E\Big[\frac{\1_{\{a_{ij}=a_{jk}=a_{ki}=1\}}}{(\tilde{d}_i^{(j,k)}+2)(\tilde{d}_j^{(i,k)}+2)(\tilde{d}_k^{(i,j)}+2)}\Big]\notag\\
&=\sum\limits_{\substack{i,j,k=1\\i,j,k\text{ p.d.}}}^{n+1}p^3\EW{\frac{1}{\tilde{d}_i^{(j,k)}+2}}^3\approx\sum\limits_{\substack{i,j,k=1\\i,j,k\text{ p.d.}}}^{n+1}p^3\mu^3
= (n+1)n(n-1)p^3\mu^3\label{propsplit2}.
\end{align}
by applying Proposition \ref{prop:asymptindep}.
As for the first summand in \eqref{propsplit1} we see that
\[\E\Big[\Big(\sum\limits_{i=1}^{n+1}z_{ii}\Big)^2\Big]=\sum\limits_{\substack{i,j,k=1\\i,j,k\text{ p.d.}}}^{n+1}\sum\limits_{\substack{i',j',k'=1\\i',j',k'\text{ p.d.}}}^{n+1}\underbrace{\E\big[\frac{a_{ij}a_{jk}a_{ki}}{d_id_jd_k}\frac{a_{i'j'}a_{j'k'}a_{k'i'}}{d_{i'}d_{j'}d_{k'}}\big]
}_{=\circledast}.\]
We differentiate the following cases, in any case we apply Proposition \ref{prop:asymptindep}:
\begin{enumerate}[leftmargin=*]
	\item If $|\{i,j,k\}\cap\{i',j',k'\}|=3$ (there are $6(n+1)n(n-1)$ possibilities for this case), then $\circledast$ is given by
	\[\E\Big[\frac{a_{ij}a_{ik}a_{jk}}{d_i^2d_j^2d_k^2}\Big]=
p^3\E\Big[\Big(\frac{1}{\tilde{d}_i^{(j,k)}+2}\Big)^2\Big]^3\approx p^3(\sigma^2)^3=p^3\sigma^6.\]
\item If $|\{i,j,k\}\cap\{i',j',k'\}|=2$ ($(n+1)n(n-1)(n-2)\cdot3\cdot2$ possibilities), then
\[
\circledast=\E\Big[\frac{a_{ij}a_{ik}a_{jk}a_{i'j}a_{i'k}}{d_id_{i'}d_j^2d_k^2}\Big]=p^5
\E\Big[\Big(\frac{1}{\tilde{d}_i^{(i',j,k)}+3}\Big)\Big]^2
\E\Big[\Big(\frac{1}{\tilde{d}_j^{(i,i',k)}+3}\Big)^2\Big]^2\approx p^5\mu^2\sigma^4.\]
\item If $|\{i,j,k\}\cap\{i',j',k'\}|=1$ ($(n+1)n(n-1)(n-2)(n-3)\cdot3$ possibilities), then 	\[\circledast=\E\Big[\frac{a_{ij}a_{ik}a_{jk}a_{i'j}a_{i'k}a_{j'k}}{d_id_{i'}d_jd_{j'}d_k^2}\Big]=
    p^6\E\Big[\Big(\frac{1}{\tilde{d}_i^{(i',j,k)}+3}\Big)\Big]^4\E\Big[\Big(\frac{1}{\tilde{d}_j^{(i,i',k)}+3}
    \Big)^2\Big]\approx p^6\mu^4\sigma^2.\]
\item If $|\{i,j,k\}\cap\{i',j',k'\}|=0$ $(n+1)\cdots(n-4)$ possibilities), then $\circledast$ is given by
    \[\circledast=\E\Big[\frac{a_{ij}a_{ik}a_{jk}a_{i'j}a_{i'k}a_{j'k}}{d_id_{i'}d_jd_{j'}d_kd_{k'}}\Big]=
    p^6\E\Big[\Big(\frac{1}{\tilde{d}_i^{(i',j,k)}+3}\Big)\Big]^6\approx p^6\mu^6.\]
\end{enumerate}
Hence:
\begin{align}
\E&\Big[\big(\sum\limits_{i=1}^{n+1}z_{ii}\big)^2\Big]=6(n+1)n(n-1)(p^3\sigma^6) + 6(n+1)n(n-1)(n-2)( p^5\mu^2\sigma^4) \notag\\
&+ 3(n+1)n(n-1)(n-2)(n-3)(p^6\mu^4\sigma^2)+ (n+1)\cdots(n-4) (p^6\mu^6)\notag
\end{align}
From Lemma \ref{lem:firstmom} and \eqref{eq:sigmaorder} we obtain that
 $\mu^2,\sigma^2=\Theta\left(\frac{1}{(np)^2}\right)$
(more precisely, we know that $\sigma^2\leq \frac 32 \mu^2$ by Remark \ref{rem:zni} for sufficiently large $n$ and they do not only follow the same order of convergence, the factor in front of the dominating term is 1 in both cases, therefore, for large $n$, $\frac{\mu^2}{\sigma^2}\to 1$). Thus,
\begin{align}
\E\Big[\big(\sum\limits_{i=1}^{n+1}z_{ii}\big)^2\Big] &\leq (n+1)n(n-1)p^3\mu^6\left[6\cdot 1.5^3+ 6\cdot 1.5^2(n-2)p^2\right.\notag\\
&\left.\hspace{3cm}+1.5\cdot3 (n-2)(n-3)p^3+(n-2)(n-3)(n-4)p^3\right]\notag\\
&\leq (n+1)n(n-1)p^3\mu^6\left[6\cdot 1.5^3+5 (n-2)(n-3)p^3+(n-2)(n-3)(n-4)p^3\right]\notag\\
&\leq 6\cdot 1.5^3n^3p^3\mu^6+(n+1)n(n-1)(n-2)(n-3)p^6\mu^6\left[5+(n-4)\right]\notag\\
&\leq 20.25n^3p^3\sigma^6+\left[(n+1)n(n-1)\right]^2p^6\mu^6\label{propsplit3}
\end{align}
This suffices to imply the desired result: From \eqref{propsplit1}-\eqref{propsplit3}
\begin{align}
\E\Big[\Big(\frac{1}{n\theta_n}\sum\limits_{k=2}^{n+1}\lambda_k^3\Big)^2\Big]&\leq \frac{20.25n^3p^3\mu^6+\left[(n+1)n(n-1)\right]^2p^6\mu^6-2(n+1)n(n-1)p^3\mu^3+1}{n^2\theta_n^2}\notag\\
&\leq\frac{20.25(np)^3\mu^6+(np)^6\left(\mu^3-\frac{1}{(n+1)n(n-1)p^3}\right)^2}{n^2\theta_n^2}\notag\\
&=\frac{20.25}{Cn}+\frac{(np)^6\left(\mu^3-\frac{1}{(n+1)n(n-1)p^3}\right)^2(np)^3}{n}\label{eq:tbq1}
\end{align}
by $\mu\leq\frac{1}{np}$ (by definition of $\mu$) and $\theta_n^2\geq np\gamma_n^2\geq C\frac{np^2}{(np)^5}$ for some small enough constant $C$, (by \eqref{defn} and Corollary \ref{cor:betagamma}) and some elementary calculations.
Using that, by \eqref{eq:qabsch}
\begin{align*}
\Bigl|\mu^3-\frac{1}{(n+1)n(n-1)p^3}\Bigr|&\leq\frac{1}{(n+1)n(n-1)p^3}-\frac{1}{n^3p^3}\left(1-L_3e^{-np}\right)\\
&\leq\frac{1}{p^3}\frac{2}{n^3(n+1)(n-1)}
\end{align*}
(since 
for sufficiently large $n$ $L_3e^{-np}\leq\frac{1}{(n+1)(n-1)}$) we obtain
\begin{align*}
\E\Big[\Big(\frac{1}{n\theta_n}\sum\limits_{k=2}^{n+1}\lambda_k^3\Big)^2\Big]&\leq\frac{20.25}{Cn}+\frac{(np)^6\frac{1}{p^6}\left(\frac{2}{n^3(n+1)(n-1)}\right)^2(np)^3}{n}\notag\\
&\approx
\frac{20.25}{Cn}+\frac{4p^3}{n^2}\to 0\notag
\end{align*}
Thus by Chebyshev,
$\frac{1}{n\theta_n}\sum\limits_{k=2}^{n+1}\lambda_k^3\Pconv 0,$
which is the claim.
\end{proof}
The last observation, of course implies that
$\frac{O\left(\sum\limits_{k=2}^{n+1}\lambda_k^3\right)}{2n\theta_n}\Pconv 0$.

Furthermore, the decomposition of $H^i$ at the beginning of this section gives
\begin{align*}
\frac{H^i-(n-2)-2\mu^2\binom{n+1}{2}p}{2n\theta_n}=\frac{1}{n\theta_n}\left(\sum\limits_{i<j}\frac{a_{ij}}{d_id_j}-\mu^2\binom{n+1}{2}p\right)
+\frac{O\Big(\sum\limits_{k=2}^{n+1}\lambda_k^3\Big)}{2n\theta_n}.
\end{align*}
By Theorem \ref{thm:indepconv}, the first summand converges in distribution to a normally distributed random variable (depending on $p^*$), while the second summand converges to 0 in probability by the above considerations.
Combining this with Slutzky's theorem yields the assertion of Theorem \ref{main_theo}.

\begin{appendices}
\section{appendix}
\begin{lem}[cf.\cite{LoTe20a}, Lemma 1.1]\label{lem:secondmomtildephi}
For any $i\neq j$
$\EW{\tilde{\Phi}^2(i,j)}=\beta_n^2-2\gamma_n^2$.

\end{lem}

\begin{rem}[Remark on Lemma~\ref{lem:combinatorial2}]
The left hand side of the bound in the lemma can be interpreted as follows: Consider a $(n+1)\times(n+1)$ field
in which we put $4k$ stones symmetrically around the diagonal, leaving the diagonal empty and arbitrarily, otherwiese. The number of stones per cell is not limited.
We discard any configuration with a row that contains exactly one stone. Each configuration gets weighted with $p$ to the the number of non-empty cells.

\noindent
This interpretation of the problem will continue throughout the proof.
\end{rem}
\begin{proof}[Proof of Lemma~\ref{lem:combinatorial2}]
Consider the symmetric matrix $K=(k_{ij})$ with $k_{ij}\in \N_0$.
\begin{align*}
\sum\limits_{\sum\limits_{i<j}k_{ij}=2k}\prod\limits_{i<j}p^{\1_{\{k_{ij}>0\}}}\prod\limits_{i=1}^{n+1}\1_{\left\{\sum\limits_{j=i+1}^{n}k_{ij}+\sum\limits_{j=1}^{i-1}k_{ji}\neq 1\right\}}
&=\sum\limits_{\sum\limits_{j\neq i}k_{ij}=4k}p^{\sum\limits_{i<j}\1_{\{k_{ij}>0\}}}\prod\limits_{i=1}^{n+1}\1_{\left\{\sum\limits_{j\neq i}k_{ij}\neq 1\right\}}
\end{align*}
Let us condition on the number of non-empty rows $m$ (a row or column is empty, if the sum of its elements is $0$).
By the condition given by the indicators, any such row has at least two ''stones'', so that $m\leq 2k$. There are $\binom{n}{m}\cdot m!$ possibilities to choose these rows.
Due to the symmetry of $K$, there are also $m\leq 2k$ non-empty columns.
At this point, there are only $m^2$ cells in the field left to distribute the stones to, every other cell is 0. These non-empty cells are symmetric around the diagonal. There are at most $(m^2/2)^{2k}$ possibilities to distribute $2k$ stones to the cells above the diagonal. The other cells are given by symmetry.

Now consider the $p$-term: For every non-empty row $i$ 
there is at least one non-empty cell $k_{ij}>0$. Since there are exactly $m$ such rows, $\sum\limits_{i<j}\1_{\{k_{ij}>0\}}+\sum\limits_{i<j}\1_{\{k_{ji}>0\}}
\geq m.$ Furthermore, by symmetry,
$\sum\limits_{i<j}\1_{\{k_{ij}>0\}}=\sum\limits_{i<j}\1_{\{k_{ji}>0\}},$
hence
$\sum\limits_{i<j}\1_{\{k_{ji}>0\}}\geq \frac{m}{2}.$
Thus,
$p^{\sum\limits_{i<j}\1_{\{k_{ij}>0\}}}\leq p^{\frac{m}{2}}.$

Now summing over the number of non-empty rows we obtain
\begin{align*}
\sum\limits_{\sum\limits_{j\neq i}k_{ij}=4k}p^{\sum\limits_{i<j}\1_{\{k_{ij}>0\}}}\prod\limits_{i=1}^{n+1}\1_{\left\{\sum\limits_{j\neq i}k_{ij}\neq 1\right\}}
&\leq\sum\limits_{m=1}^{2k}n^m(m^2)^{2k}p^{\frac{m}{2}}
\\&\leq\sum\limits_{m=1}^{2k}(m^2)^{2k}n^\frac{m}{2}(np)^{\frac{m}{2}}
= O\left(n^k(np)^k\right),
\end{align*}
since $(m^2)^{2k}$ is constant in $n$, $n^\frac{m}{2}\leq n^\frac{2k}{2}\leq n^k$ and $(np)^\frac{m}{2}\leq (np)^k$.
\end{proof}
Similar considerations lead to
\begin{cor}\label{cor:binomialcor}
\[\sum\limits_{\sum\limits_{i\neq j}k_{ij}=2k}\prod\limits_{i\neq j}p^{\1_{\{k_{ij}>0\}}}\prod\limits_{i=1}^{n+1}\1_{\left\{\sum\limits_{j\neq i}k_{ij}\neq 1\right\}}\leq O(n^{2k}p^{2k})=O\left((np)^{2k}\right)\]
\end{cor}

Towards inverse moments of binomials, a starting point may be found with
\begin{rem}\label{rem:zni}
As studied in \cite{Znidaric}, consider a random variable $Y_n$ with
\[\WK{Y_n=i}=\frac{1}{1-(1-p)^{n}}\binom{n}{i}p^i(1-p)^{n-i},\quad i=1,\dots,n\]

Following the notations in \cite{Znidaric} we find
\begin{align*}
f_r(n)&=(1-(1-p)^{n})\EW{\frac{1}{Y_{n}^r}}=\sum\limits_{i=1}^{n}\binom{n}{i}p^i(1-p)^{n-i}\frac{1}{i^r}\\
&=np\cdot\sum\limits_{i=0}^{n-1}\binom{n-1}{i}p^i(1-p)^{(n-1)-i}\frac{1}{(i+1)^{r+1}}=np\cdot\E\Bigl[\left(\frac{1}{Y+1}\right)^{r+1}\Bigr],
\end{align*}
for a $\mathrm{Bin}(n-1,p)$-distributed random variable $Y$. Consequently,
$$\E\Bigl[\Bigl(\frac{1}{Y+1}\Bigr)^r\Bigr]=\frac{1}{np}\cdot f_{r-1}(n).$$
From equation (2.17) in \cite{Znidaric}, one immediately concludes for $r\geq 2$
\begin{equation}
\E\Bigl[\Bigl(\frac{1}{Y+1}\Bigr)^r\Bigr]=\frac{1}{(np)^{r}}\cdot\Bigl(1+\frac{(r-1)r(1-p)}{2np}+\Theta\bigl((np)^{-2}\bigr)\Bigr)\label{eq:Znidaric}
\end{equation}
\end{rem}

\begin{prop}\label{prop:centralbin}
Lets $Y\sim\mathrm{Bin}(n-1,p)$ where $p\coloneqq p_n$ is some sequence so that $np(1-p)\to c$ for $0\leq c\leq \infty$. Let $\nu\coloneqq\EW{Y}=(n-1)p$. Then for any fixed $k\in\N$
\[\EW{(Y-\nu)^{2k}}\begin{cases}
=o(1),&\text{ if }c=0\\
=O(1),&\text{ if }0<c<\infty\\
=O\left((np(1-p))^k\right)&\text{ if }c=\infty.
\end{cases}\]
\end{prop}
\begin{rem}
Note that this even treats situations we are not interested in in the context of our random graph models.
In fact, when $np(1-p)\to\infty$ 
we can immediately see that we not only have a bound of the form $O\left((np(1-p)^k\right)$, but even
the relation $\Theta\left((np(1-p))^k\right)$.
\end{rem}

\begin{proof}
We interpret $Y$ as a sum of Bernoulli random variables: Let $Y_1,\dots,Y_{n-1}$ be  i.i.d. $\mathrm{Ber}(p)$-distributed random variables. Then
\begin{align*}
\E\Big[(Y-\nu)^{2k}\Big]&=\E\Big[\Big(\sum\limits_{i=1}^{n-1}Y_i-\nu\Big)^{2k}\Big]=\E\Big[\Big(\sum\limits_{i=1}^{n-1}\Big(Y_i-p\Big)\Big)^{2k}\Big]\\
&=\E\Big[\sum\limits_{k_1+\dots+k_{n-1}=2k}\binom{2k}{k_1,\dots,k_{n-1}}
\prod\limits_{i=1}^{n-1}\Big(Y_i-p\Big)^{k_{i}}\Big]\\
&\leq C_k\sum\limits_{k_1+\dots+k_{n-1}=2k}\prod\limits_{i=1}^{n-1}\EW{\Big(Y_i-p\Big)^{k_{i}}}\eqqcolon C_k\sum\limits_{k_1+\dots+k_{n-1}=2k}Z_{(k_1,\dots,k_{n-1})}
\end{align*}
by bounding the multinomial coefficient by $C_k$ and using independence of the $Z_{ij}$
For the expectation, we immediately obtain
\[\E\Big[\Big(Y_i-p\Big)^{k_{i}}\Big]
=\begin{cases}
1,&\text{if }k_i=0\\
0,&\text{if }k_i=1\\
p(1-p)^{k_i}+(1-p)(-p)^{k_i},&\text{if }k_i\geq 2
\end{cases}
\quad \leq\quad \begin{cases}
1,&\text{if }k_i=0\\
0,&\text{if }k_i=1\\
2p(1-p),&\text{if }k_i\geq 2
\end{cases}\]
From this we can immediately extract that
$Z_{(k_1,\dots,k_{n-1})}=0$
if for any $i=1,\dots,n$ we have $k_i=1$.
Otherwise, we have that
\[Z_{(k_1,\dots,k_{n-1})}=\left(2p(1-p)\right)^m\]
where $m$ is the number of indices $i=1,\dots,n$ where $k_i\neq 0$ (and as we just excluded the case $k_i\neq 1$ for any $i$, this also implies $k_i\geq 2$).
We remark that $1\leq m\leq k$. Indeed, if $m>k$, then there would be at least $k+1$ indices such that $k_i\geq 2$. Then $k_1+\dots+k_{n-1}>2k$, which contradicts the index in the sum.

Let us count the number of cases where $k_i\neq 1$ for all $i$ and the number of $k_i\neq 0$ is exactly $m$.
First, choose the indices for which $k_i\neq 0$ holds: There are $\binom{n-1}{m}m!$ possibilities for that. To begin with, each of these indices is now 2 (as it cannot be 1). Then $k_1+\dots+k_{n-1}=2m\leq 2k$. We distribute the remaining $2k-2m$ arbitrarily among the $m$ indices with positive $k_i$. For that, there are another $m^{2k-2m}$ possibilities.
We therefore have
\begin{align*}
&\EW{(Y-\nu)^{2k}}\leq C_k\sum\limits_{k_1+\dots+k_{n-1}=2k}Z_{(k_1,\dots,k_{n-1})}\\
&=C_k\sum\limits_{m=1}^k\binom{n-1}{m}m!m^{2k-2m}\cdot \left(2p(1-p)\right)^m\\
&\leq C_k\sum\limits_{m=1}^kn^mk^{2k}2^k\cdot \left(p(1-p)\right)^m= C_kk^{2k}2^k\sum\limits_{m=1}^k \left(np(1-p)\right)^m
\end{align*}
Now the factors in front of the sum are constant in $n$. 

If $np(1-p)\to 0$, The entire sum converges to 0. 

If $np(1-p)\to c$, $0<c<\infty$, then each summand is bounded and so is the sum.

Finally, if $np(1-p)\to\infty$, then $(np(1-p))^m\leq (np(1-p))^k$ for all $1\leq m\leq k$.
Hence the sum and therefore the entire term are bounded by a constant (only depending on $k$) times $(np(1-p))^k$. This yields the claim.
\end{proof}

\begin{cor}\label{cor:invmoments}
As a direct consequence, let $Y\sim\mathrm{Bin}(n-1,p)$, where $np(1-p)\to\infty$ and $\mu=\EW{\frac{1}{Y+1}}=\frac{1}{np}\left(1-(1-p)^{n}\right)$. Then
\[\E\Big[\Big(\frac{1}{Y+1}-\mu\Big)^{k}\Big]=\Theta\Big(\frac{1}{(np)^{1.5k}}\Big)\]
\end{cor}
\begin{proof}
From the definition of $\mu$ we notice with $(a+b+c)^k\leq 3^k(a^k+b^k+c^k)$
\begin{align}
\E\Big[\Big(\frac{1}{Y+1}-\mu\Big)^k\Big]&=
\E\Big[\Big(\frac{1}{Y+1}-\frac{1}{(n-1)p}+\frac{1}{(n-1)p}-\frac{1}{np}+\frac{1}{np}-\mu\Big)^k\Big]\notag\\
&\leq3^k\Big(\E\Big[\Big(\frac{1}{Y+1}-\frac{1}{\nu}\Big)^k\Big]+
\Big(\frac{1}{n(n-1)p}\Big)^k+\Big(\frac{1}{np}-\mu\Big)^k\Big)\notag\\
&\leq3^k\Big(\E\Big[\Big(\frac{1}{Y+1}-\frac{1}{\nu}\Big)^k\Big]+\frac{2}{(np)^{2k}}\Big)\label{eq:munptransf}
\end{align}
as the second and third term are each bounded by $\frac{1}{(np)^{2k}}$.
By Cauchy-Schwarz
\begin{align*}
\E\Big[\Big(\frac{1}{Y+1}-\frac{1}{\nu}\Big)^k\Big]&=\E\Big[\Big(\frac{\nu-(Y+1)}{\nu(Y+1)}\Big)^k\Big]\\
&\leq\sqrt{\E\Big[\Big(Y-\nu+1)\Big)^{2k}\Big]\E\Big[\Big(\frac{1}{\nu(Y+1)}\Big)^{2k}\Big]}
\\
&=\sqrt{\Big(\sum\limits_{m=0}^{2k}\binom{2k}{m}\E\Big[\Big(Y-\nu\Big)^{m}\Big]\Big)
\Big(\frac{1}{\nu}\Big)^{2k}\E\Big[\Big(\frac{1}{Y+1}\Big)^{2k}\Big]}\\
&=\sqrt{\Big(\sum\limits_{m=0}^{2k}\binom{2k}{m}\E\Big[\Big(Y-\nu\Big)^{2k}\Big]\Big)
\Big(\frac{1}{\nu}\Big)^{2k}\E\Big[\Big(\frac{1}{Y+1}\Big)^{2k}\Big]}
\end{align*}
by the binomial theorem and monotonicity. By Proposition \ref{prop:centralbin}, the first expectation is of order $(np(1-p))^k$. Furthermore, $\nu$ is of order $np$, and the order of the second expectation is known from Remark \ref{rem:zni}. Therefore
\begin{align*}
\E\Big[\Big(\frac{1}{Y+1}-\frac{1}{\nu}\Big)^k\Big]
&=\sqrt{\Big(\sum\limits_{m=0}^{2k}\binom{2k}{m}O\Big((np(1-p))^k\Big)\Big)O\Big(\frac{1}{(np)^{2k}}\Big)O\Big(\frac{1}{(np)^{2k}}\Big)}
\\&
=\sqrt{O\Big((np)^k\Big)}O\Big(\frac{1}{(np)^{2k}}\Big)=O\Big(\frac{1}{(np)^{1.5k}}\Big).
\end{align*}
The other terms in (\ref{eq:munptransf}) are of lower order, therefore
\[\E\Big[\Big(\frac{1}{Y+1}-\mu\Big)^{k}\Big]= O\Big(\frac{1}{(np)^{1.5k}}\Big)\]

Towards the lower bound, we first consider\[\E\Bigl[\Bigl(\frac{1}{Y+1}-\mu\Bigr)^2\Bigr]=\E\Bigl[\Bigl(\frac{1}{Y+1}\Bigr)^2\Bigr]-\mu^2=\frac{1}{np}f_1(n)-\mu^2,\]
By Lemma \ref{lem:firstmom} and Remark \ref{rem:zni}
\[\E\Bigl[\Bigl(\frac{1}{Y+1}-\mu\Bigr)^2\Bigr]=\frac{1}{np}\cdot\frac{1}{np}\Big(1+\Omega\Bigl(\frac{1}{np}\Bigr)\Big)-\frac{1}{(np)^2}=\Omega\Big(\frac{1}{(np)^3}\Big).\]
We then apply Jensen inequality:
\begin{align*}
\E\Big[\Big(\frac{1}{Y+1}-\mu\Big)^{k}\Big]&\geq\EW{\Big(\frac{1}{Y+1}-\mu\Big)^{2}}^{k/2}=\Omega\Big(\frac{1}{(np)^3}\Big)^{k/2}\\
=\Omega\Big(\frac{1}{(np)^{1.5k}}\Big)
\end{align*}
This proves the claim.
\end{proof}

\end{appendices}

%
%
\end{document}